\documentclass[12pt]{article}
\usepackage{amsfonts,amsmath,latexsym,amssymb,mathrsfs, url}

\usepackage{amsthm}

\newtheorem{theorem}{Theorem}[section]
\newtheorem{lemma}[theorem]{Lemma}

\newtheorem{remarkno}[theorem]{Remark}

\numberwithin{equation}{section}

\def\numberlikeadb{\global\def\theequation{\thesection.\arabic{equation}}}
\numberlikeadb

\newcommand{\Cov}{{\mbox{Cov}}}
\newcommand{\eqs}{\begin{eqnarray*}}
\newcommand{\ens}{\end{eqnarray*}}
\newcommand{\eqa}{\begin{eqnarray}}
\newcommand{\ena}{\end{eqnarray}}
\newcommand{\eq}{\begin{equation}}
\newcommand{\en}{\end{equation}}

\def\ignore#1{}

\def\half{{\textstyle{\frac12}}}

\def\be{{\mathbf e}}
\def\bn{{\mathbf n}}
\def\beps{{\boldsymbol\varepsilon}}

\def\ul{^{(l)}}
\def\uj{^{(j)}}

\def\bU{{\mathbf U}}
\def\bV{{\mathbf V}}

\def\Ref#1{(\ref{#1})}
\def\a{\alpha}

\def\s{\sigma}

\def\l{\lambda}
\def\L{\Lambda}

\def\law{{\cal L}}

\def\Def{\ :=\ }

\def\giv{\,|\,}

\def\tfrac#1#2{{\textstyle{\frac#1#2}}}

\def\Po{{\rm Po\,}}

\def\non{\nonumber}
\def\th{\theta}

\def\e{\varepsilon}
\def\m{\mu}

\def\var{{\rm Var}}

\def\g{\gamma}
\def\h{\eta}

\def\nti{n\to\infty}

\def\un{^{(n)}}
\def\Bl{\left(}
\def\Br{\right)}

\def\Blb{\left\{}
\def\Brb{\right\}}

\def\ui{^{(1)}}

\def\ut{^{(2)}}
\def\uh{^{(3)}}
\def\uf{^{(4)}}
\def\uv{^{(5)}}

\def\Bi{{\rm Bi\,}}

\def\nin{\noindent}
\def\D{\Delta}

\def\ex{{\mathbb E}}
\def\pr{{\mathbb P}}
\def\msk{\medskip}

\def\Be{{\rm Be\,}}
\def\cF{{\cal F}}
\def\tcF{\widetilde{\cF}}
\def\cG{{\cal G}}
\def\d{\delta}

\def\Eq{\ =\ }
\def\Le{\ \le\ }

\def\Giv{\,\Big|\,}
\def\cupdot{\cup\kern-8.2pt\cdot\kern5.5pt}
\def\cov{\Cov}
\def\corr{\mathrm{Corr}}

\def\TT{^\top}

\def\cN{{\mathcal N}}
\def\tcN{{\widetilde{\cN}}}

\def\hW{{\widehat{W}}}

\def\card{{\rm card}}
\def\tU{{\widetilde U}}
\def\btU{\mathbf{\tU}}
\def\tW{{\widetilde W}}
\def\tD{{\widetilde \D}}
\def\tZ{{\widetilde Z}}
\def\hZ{{\widehat Z}}
\def\tm{{\widetilde m}}
\def\hath{{\widehat\h}}
\def\und{'{}^{(n)}}

\def\DG{{\mathcal{DG}}}
\def\oplus{^{(o,+)}}
\def\iplus{^{(i,+)}}

\def\mm{m_1}
\def\so{^{(o)}}
\def\si{^{(i)}}
\def\sos{^{(o*)}}

\def\sb{^{(b)}}
\def\tN{{\widetilde N}}

\def\fp{\kern0.3pt{\mathfrak p}}
\def\cI{\mathcal{I}}
\def\cJ{\mathcal{J}}
\def\tcI{{\widetilde{\cI}}}
\def\hcI{{\widehat{\cI}}}
\def\hI{{\widehat{I}}}
\def\tI{{\widetilde{I}}}
\def\MN{\mathcal{MN}}
\def\gdz{}
\def\DB{\mathcal{DB}}

\def\Bmmi{\Bigl(\frac m{m-1}\Bigr)}
\def\Bmmio{\Bigl(\frac{m_0}{m_0-1}\Bigr)}
\def\Bmmith{\Bigl(\frac{m\th}{m\th-1}\Bigr)}
\def\bzero{\mathbf 0}
\def\bone{\mathbf 1}
\def\il{^{1,l}}
\def\tl{^{2,l}}
\def\ih{^{1,3}}
\def\thh{^{2,3}}

\def\baZ{Z^*}
\def\baZp{Z^{*+}}
\def\usp{^{*+}}
\def\tbaZp{\tZ^{*+}}
\def\tbaZ{\tZ^{*}}
\def\dbar{{\bar d}}
\def\tX{{\widetilde X}}
\def\baW{{\overline{W}}}
\def\jj{{j}}
\def\ii{{i}}
\def\BGW{Bienaym\'e--Galton--Watson}
\def\adbr{}
\def\adbb{}

\begin{document}

\title{Approximating inter-point distances in directed Bernoulli graphs}

\author{
A. D. Barbour\footnote{Institut f\"ur Mathematik, Universit\"at Z\"urich,
Winterthurertrasse 190, CH-8057 Z\"URICH; a.d.barbour@math.uzh.ch.
ADB was supported in part by Australian Research Council project DP220100973.
\msk}
\ and
Gesine Reinert\footnote{Department of Statistics,
University of Oxford, 24--29 St Giles, OXFORD OX1 3LB, UK.
GDR was supported in part by EPSRC grants \adbr{EP/T018445/1, EP/V056883/1, EP/Y028872/1 and EP/X002195/1}.
}\\
Universit\"at Z\"urich and University of Oxford }

\date{}
\maketitle

\begin{abstract}
In directed random graphs, in which edges can be assigned to have one of two directions, or perhaps both,
the distance between two vertices $v$ and~$v'$ can be computed along
paths that are directed from $v$ to~$v'$, or along paths that are directed from $v'$ to~$v$.  
These two distances are in general dependent.
Here, we approximate their joint distribution in the setting of the directed Bernoulli random graph~$\DG(n,p,\th)$, 
obtained as a natural extension of the Bernoulli random graph $\cG(n,p)$ by assigning directions to the edges independently, 
bidirectional with probability~$\th$, 
and either of the two possible choices of single direction with probability $\half(1-\th)$.
The approximation involves two independent copies of a trivariate limiting random vector $(W^*_1,W^*_2,W_3)$
associated with a $3$-type
Bienaym\'e--Galton--Watson process.  \adbr{The approximation error is shown to be typically of order~$O(n^{-1/2}\log n)$;
this asymptotic order is likely to be optimal, even for the corresponding approximation in the Bernoulli
random graph $\cG(n,p)$.}
\end{abstract}

\nin {\bf Keywords:}  Shortest paths, inter-point distances, Bernoulli digraphs.

\msk\nin
{\bf MRC subject classification:}

\section{Introduction}\label{intro}
 \setcounter{equation}{0}
In many networks observed in practice, typical inter-point distances are found to grow with the
logarithm of the number~$n$ of vertices, rather than with a power of~$n$, as would be the case in
spatial geometric graphs.  This is the case in the Bernoulli random graph~$\cG(n,m/n)$ for~$m > 1$ fixed, 
where typical distances between pairs of points in the giant component are comparable to $\log n/\log m$.
Replacing~$m$ by the mean vertex degree~$\dbar_n$ in a given network, the ratio $\log n/\log\dbar_n$
is a reasonable approximation to the average interpoint distance in a number of networks.
These include the  somatic neural network of {\it C.~elegans\/}, 
as considered by Pavlovic et al.~(2014), with $n = 279$;  a  coauthorship network 
from Newman~(2001), with $n = 13'861$; and the political blog network from Adamic \& Glance~(2005), with $n = 1'222$.

The asymptotic distributions of inter-point distances in locally branching graph models, such as the
Bernoulli random graph $\cG(n,m/n)$, have been established in many papers, 
including Barbour \& Reinert \adbr{(2001, 2006, 2011, 2013)},
van der Hofstad, Hooghiemstra \& Van Mieghem~(2005), van den Esker, van der Hofstad, Hooghiemstra \& Znamenski~(2005),
van den Esker, van der Hofstad \& Hooghiemstra~(2008) and van der Hoorn \&  Olvera--Cravioto~(2018).  The strategy 
in each case, in a graph on~$n$ vertices,  is to
choose two vertices $V$ and~$V'$ independently at random, and to approximate the distributions of the sizes of their
neighbourhoods out to distances~$t_n$ at which the {\it expected\/} neighbourhood sizes are roughly~$\sqrt n$.
For neighbourhoods of this relatively small size, a branching process approximation to the sizes of the
neighbourhoods is sufficiently accurate.  Given the sizes of the neighbourhoods, the probability
that they are disjoint can then also be approximated.  If the neighbourhoods of radius~$t_n$ of the vertices $V$ and~$V'$ 
are disjoint, then the graph distance $D(V;V')$ between them exceeds~$2t_n$, implying a corresponding approximation to the probability
 $\pr[D(V;V') > 2t_n]$.  
 
Based on such a strategy, we can derive the following theorem, describing the distribution of the inter-point distances
in $\cG(n,m/n)$ for large~$n$, \adbr{when $m > 1$}.  
In its statement, the random variables $W$ and~$\tW$ represent {\it independent\/} copies
of the random variable $\lim_{r \to \infty} m^{-r}Z_r$ in a Bienaym\'e--Galton--Watson process~$Z$ with $Z_0=1$, having
offspring distribution~$\Po(m)$.  The quantity
\eq\label{B-rn-def}
     r_n \Def \lfloor \half \log n/\log m \rfloor
\en
is chosen to satisfy $\ex\{Z_{r_n}\} \approx \sqrt n$, so that the expected sizes of neighbourhoods of radius close to~$r_n$
are comparable to~$\sqrt n$.
It is then the case that, \adbr{for $m > 1$,}
\eq\label{B-chin-def}
   1 \Le \chi_n \Def m^{-r_n}\sqrt n \Le m.
\en  

\begin{theorem}\label{B-dist}
Let  $V_n$ and~$V'_n$ be a pair of vertices chosen independently and uniformly at random from the vertex set of
a Bernoulli random graph $G \sim \cG(n,m/n)$, with $m > 1$.
Then, for any $n$ sufficiently large, and with $r_n$ and~$\chi_n$ defined in \Ref{B-rn-def} and~\Ref{B-chin-def}, we have
 \eqs
      \lefteqn{\biggl|\pr[D(V_n,V'_n) > 2r_n+u] - \ex\Bl \exp\Blb - m^{u+1} W \tW/\{(m-1)\chi_n^2\} \Brb \Br \biggr|} \\
            && \Le C(m) m^{u/2}\max\{1,m^u\}\,  n^{-1/2}\log n, \phantom{XXXXXXXXXXXX}
 \ens
 for any $u > -2r_n$,
where $C(m)$ is uniformly bounded in any interval of the form $1+\d \le m \le \D$, for $\d,\D > 0$.
\end{theorem}

\noindent The theorem is a refinement of that given in Barbour \& Reinert~(2026), Theorem~11.6.1, in that
the error bound improves on the one given there in Equation~(11.6.20),
where the corresponding bound is of order~$O(n^{-1/8})$.
Justification for $n^{-1/2}\log n$ being the best possible asymptotic order as $\nti$ is given in
Remark~\ref{B-optimality}.

For directed random graphs, in which edges can be assigned to have one of two directions, or perhaps both,
the distance~\adbr{$D$} between two vertices $v$ and~$v'$ can be computed along
paths that are directed from $v$ to~$v'$ ($D(v;v')$), or along paths that are directed from $v'$ to~$v$ ($D(v';v)$).  
This leads to the question of the {\it joint\/} distribution of the two distances,
something not addressed by van der Hoorn \&  Olvera--Cravioto~(2018) in their analysis of the directed configuration model.  
The two distances are in general dependent,
since the presence of an edge directed from vertex $i$ to vertex~$j$ need not be independent of the presence of
an edge between $j$ and~$i$, and this introduces its own complications.
Here, we investigate the problem in the simplest setting, that of the directed Bernoulli random graph~$\DG(n,p,\th)$, 
obtained as a natural extension of the Bernoulli random graph $\cG(n,p)$.  In this model, starting from a
realization of $\cG(n,p)$,  directions are assigned to edges independently, bidirectional
with probability~$\th$, 
and either of the two possible choices of single direction with probability $\half(1-\th)$.
\ignore{
If~$A$ denotes the corresponding random adjacency matrix, it follows that
$\pr[A_{ij}=1] = p\th + \half p(1-\th) = \half p(1 + \th)$ and that $\pr[A_{ij} = A_{ji} = 1] = p\th$, for any pair $(i,j)$.
Hence the {\it conditional\/} probability $\pr[A_{ij} = 1 \giv A_{ji} = 1]$ is given by
\eq\label{DG-cdl-edge-prob}
    \pr[A_{ij} = 1 \giv A_{ji} = 1] \Eq \frac{2\th}{1 + \th}\,,
\en
giving another interpretation to~$\th$.
In the particular instance of the model in which
 all the elements $(A_{ij},\,1 \le i \neq j \le n)$ are {\it independent\/}, and have the same
probability~$\fp$, it follows that $\pr[A_{ij} = 1 \giv A_{ji} = 1] = \pr[A_{ij}=1]$, and hence that
\eq\label{ER-rel-DG-fracp-def}
      \frac{2\th}{1 + \th} \Eq \half p(1 + \th) \Eq \fp.
\en
Thus, in this restricted model, the directed Bernoulli graph~$\DB(n,\fp)$ with edge probability~$\fp$,
the quantities $\th$ and~$p$ are related by the equation
\[
    \frac{4\th}{(1 + \th)^2} \Eq p, 
\]
and both can be recovered from~$\fp$.
}

Fix any $v \in [n]$, and consider first the structure of the out-neighbourhood rings~$\cI^o_r(v)$, $r\ge0$,
consisting of those points that can be reached from the vertex~$v$ by a path of length~$r$ consisting only of
edges that are either bidirectional or directed away from~$v$ along the path, but cannot be reached along any such path of
shorter length; \adbr{set $\cI^o_0(v) = \{v\}$}.  Define $\cN^o_r(v) := \cup_{s=0}^r \cI^o_s(v)$ to be the $r$-out-neighbourhood of~$v$,
and set $I^o_r(v) := \card(\cI^o_r(v))$ and $N^o_r(v) := \card(\cN^o_r(v))$, \adbr{with $I^o_0(v)=1$}.
Let $p_0 := p\th + \half p(1-\th) =\half p(1+ \th)$ denote the probability that, for any given pair $(i,j)$,
there is a directed (possibly bidirectional) edge from $i$ to~$j$.
Then the sizes of the out-neighbourhood rings can be constructed by taking 
$I_r = I_r^o(v)$, $S_r = n - N^o_{r}(v) = n - \sum_{s=0}^r I_s$
and $p = p_0$ in the Reed--Frost recursion 
from the theory of epidemics \adbr{(see, for example, Bailey~(1990), Chapter 12.5)}
that specifies the {\it conditional distribution\/} of~$(I_{r+1},S_{r+1})$, given that the values
of~$S_r$ and~$I_r$ are known: 
\eq\label{Reed-Frost-recursion}
     \law\bigl(I_{r+1} \giv (S_l,I_l), 0 \le l \le r\bigr) \Eq \law(I_{r+1} \giv I_r,I_{r-1},\ldots,I_0) \Eq \Bi(S_r,1 - (1-p)^{I_r}),
\en
with $S_{r+1} = S_r - I_{r+1}$.
The explanation for this is that, given $I_0,\ldots,I_r$,
there are~$I_r$ potential edges between the set of vertices at distance~$r$
from~$v$ and any member~$u$ of the set of $S_r$ vertices at distance more than~$r$
from~$v$, and at least one of them has to be present, and directed outward from~$v$ or bidirectional,
if~$u$ is to be at distance~$(r+1)$ from~$v$. 
Thus the joint distribution of the sizes of the out-neighbourhood rings $(I_l^o(v),\,0\le l\le r)$ is the same as
that of the neighbourhood rings $(I_l(v),\,0\le l\le r)$ in a Bernoulli random graph $\cG(n,p_0)$ with edge probability~$p_0$.
\ignore{As before, we consider the asymptotics
when~$n$ increases and $p = m/n$, for~$m$ fixed.  Defining
\eq\label{ER-rel-DG-mo}
    m_0 \Def \half m(1+\th) \adbd{\Eq np^o},
\en
it then follows from Theorem~\ref{RF-2}, with $r_0 =\a\log n/\log m_0$, that, for any $0 < \a < 2/3$, the
distribution of the out-neighbourhood ring \index{neighbourhood ring process!out-neighbourhood}
process~$(I^o_r(v),\,1\le r\le \a\log n/\log m_0)$ is asymptotically close in total variation
to that of $(Z_r,\,1\le r\le \a\log n/\log m_0)$, where~$Z$ is the Poisson branching process
with offspring mean~$m_0$ starting with $Z_0 = 1$.  \gdz{As in Theorem~\ref{RF-2}, the} upper limit on the range of~$r$
can be dispensed with if $m_0 \le 1$.  
}
The same argument 
can also be used for the in-neighbourhood rings~$\cI^i_r(v)$, $r\ge0$,
\adbr{defined in the same way as the \adbb{out-neighbourhood} rings, but for paths consisting only of edges
 that are either bidirectional or directed towards~$v$ along the path.}
Thus, for any fixed~$v$, the local neighbourhood structures of the in- and out-neighbourhoods, each considered in isolation,
are in distribution the same as for~$\cG(n,p_0)$.
These observations can be used to show that the distribution of the {\it directed\/} distance $D(V;V')$ between two randomly
chosen vertices $V$ and~$V'$ in the model $\DG(n,m/n,\th)$ can be derived from Theorem~\ref{B-dist}, with~$m$ in the
statement of the theorem replaced by $m_0 := \half m(1+\th)$, \adbr{as long as $m_0 > 1$}.

However, the in- and out-neighbourhood processes are in general dependent.  For instance, the (unconditional)
distribution of $I_1^o(v)$ is $\Bi(n-1,p_0)$, whereas, conditional on the event $I_1^i(v) = 0$,
the distribution of~$I_1^o(v)$ is $\Bi(n-1,p')$, where, using the independence of edges and Bayes theorem,
\[
     p' \Eq \frac{\half p(1-\th)}{1 - p + \half p(1- \th)}\,.
\]
\ignore{
These two distributions are only the same if}
\adbr{Thus, unless $p'  = p_0 = \half p(1+\th)$, the unconditional distribution of $I_1^o(v)$ is different from 
its distribution conditional on the event $I_1^i(v) = 0$, implying that  the in- and out-neighbourhood processes 
are dependent.}  As a result, the joint distribution of $D(V;V')$ and~$D(V';V)$ in general needs more elaborate 
consideration, and is addressed in Section~\ref{s-digraph}; see Theorem~\ref{DG-dist}.
\ignore{If $p'  = p_0 = \half p(1+\th)$, it follows that $p = 4\th/(1+\th)^2$, and hence that $p_0 = p' = 2\th/(1+\th)$.
For such pairs~$(p,\th)$, each possible {\it directed\/} edge is present independently with probability
$p' = 2\th/(1+\th)$, 
giving the family of sub-models $(\DB(n,p'), 0 \le p' \le 1)$, where $\DB(n,p') = \DG(n,p,\th)$,
for~$(p,\th)$ the solution to the equations $2\th/(1+\th) = p'$ and $p\th = (p')^2$.}

\section{The Bernoulli random graph}\label{s-Bernoulli}
In this section, we give a proof of Theorem~\ref{B-dist} that is relatively simple,
and introduces methods that are used also for the Bernoulli digraph.

For any vertex~$v$, let~$\cI_r(v)$ denote the set of vertices in~$G$ that are at graph distance
exactly~$r$ from~$v$,  the $r$-neighbourhood ring around~$v$, with $\cI_0(v) = \{v\}$, and write
$\cN_r(v) := \bigcup_{s=0}^r \cI_s(v)$, the $r$-neighbourhood of~$v$.  Define
$I_r(v) := \card(\cI_r(v))$ to be the number of vertices in~$G$ at graph distance equal to $r$ from~$v$, and write
$N_r(v) := \sum_{s=0}^r I_s(v)$ for the number of vertices at graph distance less than or equal to $r$ from $v$.
We now consider the development of the process $(I_r(v),\,r\ge0)$.  We begin by comparing~$I_r(v)$ to the number
of individuals in
the $r$-th generation of a branching process~$Z\un$ with $Z\un_0 = 1$ and with offspring distribution~$\Bi(n-1,m/n)$.

\begin{lemma}\label{L1}
It is possible to construct a copy of the process  $(I_r(v),\,r\ge0)$ and a branching process $Z\un := (Z\un_r,\,r\ge0)$,
with $Z\un_0 = 1$ and with offspring distribution~$\Bi(n-1,m/n)$, on the same probability space,
in such a way that $Z\un_r \ge I_r(v)$ almost surely for all $r \ge 0$.
\end{lemma}

\begin{proof}
Let $(V_{ijk},\,i,j,k \ge 1)$ be independent indicator random variables with $\pr[V_{ijk}=1] = m/n$ for all $i,j,k$.
Then $U_{ij} := \sum_{k=1}^{n-1} V_{ijk}$, $i,j \ge 1$, are independent random variables with distribution~$\Bi(n-1,m/n)$.
Let~$Z\un$ be the branching process constructed recursively by taking $Z\un_0 = 1$ and, for $r \ge 1$, by taking
$Z_r := \sum_{j=1}^{Z_{r-1}} U_{rj}$; this branching process has offspring distribution~$\Bi(n-1,m/n)$.
A copy of the process  \adbr{$(I_r(v),\,r\ge0)$}
is also constructed recursively from the~$V_{ijk}$ by taking $I_0(v) = 1$ and $\cI_0(v) = \{v\}$, and, for $r \ge 1$,
by taking $\tU_{rj} :=  \sum_{k=1}^{n - N_{r-1}(v)} V_{rjk}$,  for $1 \le j \le I_{r-1}(v)$,
 to represent the number of edges joining the $j$-th vertex in~$\cI_{r-1}(v)$ (with any prescribed ordering)
to vertices outside~$\cN_{r-1}(v)$.
For each $1 \le j \le I_{r-1}(v)$, $\tU_{rj}$ vertices can now be chosen at random (without replacement) from
$[n] \setminus \cN_{r-1}(v)$, to constitute the vertices outside $\cN_{r-1}(v)$ that are joined to the $j$-th vertex in~$\cI_{r-1}(v)$;
the union over~$j$ of these sets of vertices yields~$\cI_r(v)$.
Note that thus $I_r(v) \le \sum_{j=1}^{I_{r-1}(v)} \tU_{rj} \le \sum_{j=1}^{I_{r-1}(v)} U_{rj}$, and that
thus $I_r(v) \le Z\un_r$ if $I_{r-1}(v) \le Z\un_{r-1}$;  induction starting with $Z_0 = I_0(v) = 1$ thus shows that
$I_r(v) \le Z\un_r$ for all $r \ge 0$.
\end{proof}

\ignore{
to be constructed jointly with the process $(N_r(v),\,r\ge0)$.
For $r \ge 0$, let $\cF_r(v)$ denote the $\s$-algebra generated by $((\cN_s(v),Z\un_s)\,0 \le s \le r)$.
Taking any vertex $v' \in \cN_r(v)$, note that, given~$\cF_r(v)$, the distribution of the number of vertices
connected to~$v'$ that
do not belong to $\cN_r(v)$ has distribution $\Bi(n - N_r(v)-1,m/n)$.  This distribution is stochastically smaller
than~$\Bi(n-1,m/n)$, as can be seen by constructing a random variable $X \sim \Bi(n-1,m/n)$ by first realizing
a random variable $X' \sim \Bi(n - N_r(v)-1,m/n)$, and then adding to it an independent copy of~$\Bi(N_r(v),m/n)$.
Now, for any $r \ge 0$, consider the construction of~$\cN_{r+1}(v)$, given~$\cF_r(v)$, and on the event $\{N_r(v) \le Z\un_r\}$.  
Conditional on~$\cF_r(v)$, the vertices $v' \in \cN_r(v)$ are
connected to numbers of vertices in $[n] \setminus \cN_r(v)$ that are independent of one another, and have
distributions stochastically smaller than~$\Bi(n-1,m/n)$.  Hence the total number~$U_{r+1}(v)$ of such connections is no larger than
the total number of offspring of~$N_r(v)$ individuals in one generation of a branching process~$Z\un$ with
offspring distribution~$\Bi(n-1,m/n)$, and thus, given $N_r(v) \le Z\un_r$, it follows that~$U_{r+1}(v)$ has conditional distribution
stochastically smaller that that of~$Z\un_{r+1}(v)$, given~$\cF_(v)r$.   Hence, for example using the quantile coupling,
it is possible to realize $Z\un_{r+1}$ and~$U_{r+1}(v)$ together on the same probablilty space, in such a way that 
$Z\un_{r+1} \ge U_{r+1}(v)$.  Since $N_{r+1}(v) \le U_{r+1}(v)$, because 
each vertex in~$\cN_{r+1}(v)$ must be connected to at least one vertex in~$\cN_r(v)$, it thus follows that, with this coupling,
$Z\un_{r+1} \ge N_{r+1}(v)$ also.   Thus, repeating this construction for each $r \ge 0$, and recalling that $Z\un_0 = N_0(v) = 1$,
it follows that $Z\un_r \ge N_r(v)$ almost surely for all $r \ge 0$.
}

The conclusion of Lemma~\ref{L1} is useful for bounding elements in the approximation errors that arise later in the argument.
However, it is also important to be able to show that $Z\un_r$ and~$I_r(v)$, as constructed in Lemma~\ref{L1}, are not too far apart.
Let $\cF_r(v)$ denote the $\s$-algebra generated by $\{(\cI_s(v),Z\un_s),\,0 \le s \le r\}$.

\begin{lemma}\label{L2}
The processes $(Z\un_r,\,r\ge0)$ and $(I_r(v),\,r\ge0)$, as constructed in the proof of Lemma~\ref{L1},
are such that
\[
   \ex |Z\un_r - I_r(v)| \Le  (m/n)\h'_0(r,m) \Le n^{-1}  m^{2r}\,\Bmmi^3\,, \qquad r \ge 0,
\]
where $\h'_0(r,m)$, defined in~\Ref{h0-dash-def}, is increasing in both $r$ and~\adbr{$m > 1$}.
\end{lemma}

\begin{proof}
 Fix any $r \ge 1$.  Then 
 the conditional probability, given~$\cF_r(v)$, that a vertex in
 $[n] \setminus \cN_r(v)$  belongs to $\cI_{r+1}(v)$ is $1 - (1 - m/n)^{I_r(v)}$.  
 Since $sx - \binom s2 x^2 \le 1 - (1-x)^s \le sx$ in $x \ge 0$, it thus follows,
 taking $x = m/n$ and $s = I_r(v)$, that 
 \[
        (n - N_r(v)) \frac{m I_r(v)}n\,\Bl 1 -  \frac{m I_r(v)}{2n}\Br  \Le \ex\{I_{r+1}(v) \giv \cF_r(v)\} \Le   m_n I_r(v),
 \]
 where $m_n := (n-1)m/n < m$, using $N_r(v) \ge 1$. Writing $n_r := \ex\{I_r(v)\}$ and taking expectations, this in turn implies that
 \eq\label{2.1}
      m_n n_r - \e_n(r,m) \Le n_{r+1} \Le m_n n_r,
 \en
where
\eq\label{eps-def}
  \e_n(r,m) \Def (m/n) \bigl(\ex\{I_r(v) N_r(v)\} + \tfrac12 m\ex\{(I_r(v))^2\} \bigr).
\en
Iterating~\Ref{2.1}, and because $n_0 = 1$, we deduce that, for all $r \ge 0$,
\eq\label{2.2}
    m_n^r - \h_{0,n}(r,m) \Le n_r \Le m_n^r,
\en
where
\eq\label{h0-def}
    \h_{0,n}(r,m) \Def \sum_{j=0}^{r-1} m_n^{r-1-j} \e_n(j,m). 
\en

Now, because the variance of the distribution \hbox{$\Bi(n-1,m/n)$} is less
than~$m$, it follows from \adbr{Harris~(1963), Chapter~1, Theorem~5.1}, 
that
\eq\label{variance-bnd}
    \ex\{(Z\un_s)^2\} \Le \frac{m^s(m^{s+1}-1)}{m-1} \Le m^{2s}\Bmmi, \quad m > 1,\ \adbr{s \ge 0},
\en
where the first bound is increasing in both $s$ and~$m$ in $m > 1$.
Hence, since $\ex\{Z\un_r \giv Z\un_s\} = m_n^{r-s}Z\un_s$ in $r \ge s$, it follows
from Lemma~\ref{L1}, 
and using $m_n \le m$, that
\eq\label{IsIr-bnd}
\begin{split}
     \ex\{I_s(v)I_r(v)\} &\Le \ex\{Z\un_s Z\un_r\} \Eq m^{r-s}\ex\{(Z\un_s)^2\} \\
      &\Le m^r \,\frac{m^{s+1}-1}{m-1}\ =:\ m^r\h_1(s,m),\quad \adbr{r \ge s},
\end{split}
\en
where 
\eq\label{eta1-bnd}
   \h_1(s,m) \Def \frac{m^{s+1}-1}{m-1} \Le m^s\,\frac m{m-1}
\en
is increasing in both $s$ and~$m$ in $m > 1$.
Hence, with a little calculation based on~\Ref{IsIr-bnd},  it follows that
\eqa
   \ex\{N_r^2(v)\} &=& \sum_{s=0}^r \ex\{I_s^2(v)\} + 2\sum_{0 \le s < t \le r}\ex\{I_s(v)I_t(v)\}  \non\\
                   &\le& \h_2(r,m) \Def \sum_{s=0}^r m^{s}\h_1(s,m) + 2\sum_{0 \le s < t \le r} m^t \h_1(s.m) \label{N-squared-bnd-0}\\
                   &\le& m^{2r}\Bmmi^3\,, \label{N-squared-bnd}
\ena
where $\h_2(r,m)$ is increasing in both $s$ and~$m$ in $m > 1$.
It then follows from \Ref{IsIr-bnd} and~\Ref{N-squared-bnd}, using Cauchy--Schwarz, that
\eqa
    (n/m)\e_n(r,m) &=& \ex\{I_r(v) N_r(v)\} + \tfrac12 m\ex\{(I_r(v))^2\} \non\\
   &\le& \h_3(r,m) \Def \sqrt{m^r\h_1(r,m) \h_2(r,m)} + \tfrac12 m^r\h_1(r,m) \label{IN+mI-squared-0}\\
   &\le& m^{2r}\Bmmi^2 + \frac{m^{2r+2}}{2(m-1)} \Le m^{2r+1}\,\Bmmi^2, \label{IN+mI-squared}
\ena
 where $\h_3(r,m)$ 
is increasing in both $r$ and~$m$ in $r \ge 1$ and $m > 1$.
Hence, from~\Ref{h0-def}, we have
\[
   \h_{0,n}(r,m) \Le \frac mn \h'_0(r,m).
\]
where
\eq\label{h0-dash-def}
    \h'_0(r,m) \Def \sum_{j=0}^{r-1} m^{r-1-j} \h_3\un(\adbr{j},m) \Le m^{2r-1}\,\Bmmi^3
\en
is increasing in both $r$ and~$m$ in $r \ge 1$ and $m > 1$.
\ignore{
\eq\label{2.2}
    m_n^r - \h_{0,n}(r,m) \Le n_r \Le m_n^r,
\en
where
\eq\label{h0-def}
    \h_{0,n}(r,m) \Def \sum_{j=0}^{r-1} m_n^{r-1-j} \e_n(j,m). 
\en
Thus, iterating~\Ref{2.1}, and because $n_0 = 1$, we deduce that, for all $r \ge 0$,
\eq\label{2.2}
    m_n^r - (m/n)\h_2'(r,m) \Le n_r \Le m_n^r,
\en
where
\eq\label{h-dash-def}
    \h_2'(r,m) \Def \sum_{j=0}^{r-1} m_n^{r-1-j} \h_2(j,m) \Le m^{2r-1}\,\Bigl(\frac m{m-1}\Bigr)^3 
\en
is increasing in both $r$ and~$m$ in $m > 1$.
}

Since $\ex Z\un_r = m_n^r$ and $Z\un_r \ge I_r(v)$ a.s.\ for all~$r$,
it thus follows that $\ex|Z\un_r - I_r(v)| \le (m/n)\h_0'(r,m)$ for all~$r \ge 0$, proving the lemma.
\end{proof}

The fact that the approximating branching process~$Z\un$ is different for each~$n$ is inconvenient for the
asymptotics.  Instead, we prefer to use a branching process~$Z$ with $Z_0=1$ and with offspring distribution
the Poisson distribution~$\Po(m)$ with mean~$m$.  A simple argument now shows that this makes relatively little difference.

\begin{lemma}\label{L3}
For any $n \ge m \max\{2,1/(\sqrt m - 1)\}$,
the processes $(I_r(v),\,r\ge0)$ and $(Z_r,\,r\ge0)$ can be constructed on the same probability space in such a way
that
\[
   \ex|Z_r - I_r(v)| \Le (m/n)\h(r,m),\qquad r \ge 0,
\]
where $\h(0,m) = 0$ and
\eq\label{eta-bnd}
    \h(r,m) \Le  \frac{m^{2r+2}}{(m-1)^3} + 3r m^{3r/2} \Le 4m^{2r}\,\Bmmi^3, 
               \qquad r\ge1,
\en
is increasing in both $r$ and~$m$ in $m > 1$.
\end{lemma}

\begin{proof}
\adbr{First, couple $(I_r(v),\,r\ge0)$ to $(Z\un_r,\,r\ge0)$, as in Lemma~\ref{L1}.  Next,}
 let~$\l_n$ be chosen so that 
 \[
    e^{-\l_n} \Eq \Po(\l_n)\{0\} \Eq \Be(m/n)\{0\} \Eq (1-m/n);
 \]
note that then, for $n \ge 2m$, 
\eq\label{3.1}
     m/n \Le \l_n \Eq -\log(1-m/n) \Le (m/n)(1 + m/n).
\en
With this choice of~$\l_n$, the distribution $\Bi(n-1,m/n)$ is stochastically smaller than the Poisson distribution~$\Po(\tm_n)$
with mean~$\tm_n := (n-1)\l_n$.
Hence~$Z\un$, \adbr{whose offspring distribution is $\Bi(n-1,m/n)$,} 
can be coupled to a branching process~$\tZ\un$ with $\tZ\un_0 = 1$ and with offspring distribution~$\Po(\tm_n)$ in
such a way that $Z\un_r \le \tZ\un_r$ for each $r \ge 0$;  the argument of Lemma~\ref{L1} can be used in essence,
realizing independent pairs of offspring numbers coupled so that
those of~$Z\un$ are a.s.\ always less than or equal to those of~$\tZ\un$.  Hence it follows that, with this
coupling, $\ex|\tZ\un_r - Z\un_r| \le \tm_n^r - m_n^r$ for each $r \ge 0$.

Then the branching process~$\tZ\un$ can be coupled in
a similar monotone way (depending on which of $\tm_n$ and~$m$ is larger) to a branching process~$Z$ with $Z_0=1$ and with 
offspring distribution~$\Po(m)$, 
because the Poisson distribution~$\Po(\m)$ is stochastically increasing in~$\m$; with such a coupling, it follows that
$\ex|\tZ\un_r - Z_r| \le |\tm_n^r - m^r|$.  It thus follows, by the triangle inequality, that
\[
   \ex|Z\un_r - Z_r| \Le  2(\tm_n^r - m_n^r) + m^r - m_n^r.
\]
Since, by~\Ref{3.1}, for $n \ge m/(\sqrt m - 1)$, ensuring that $1 + m/n \le \sqrt m$, we have
\eq\label{3.69}
   0 \Le \tm_n^r - m_n^r \Le  m_n^r\{(1 + m/n)^r - 1\} \Le \frac{rm}n\,m^{3r/2},
\en
and since $m^r - m_n^r = m^r\{1 - (1 - 1/n)^r\} \le rm^r/n$,
it follows that
\eq\label{3.7}
  \ex|Z\un_r - Z_r| \Le  \frac{r}n\,\{2m^{1+3r/2} + m^r\}.
\en

As a result, by the triangle inequality, and using Lemma~\ref{L2}, it follows that, for $r \ge 1$,
\eqs
   \ex|Z_r - I_r(v)| &\le& \ex|Z\un_r - I_r(v)| + \ex|Z_r - Z\un_r| \\
       &\le&  (m/n)\bigl\{\h_0'(r,m) + r\{2m^{3r/2} + m^{r-1}\}\bigr\} \ =:\ (m/n)\h(r,m),
\ens
where $\h(r,m)$ is increasing in both $r$ and~$m$, and is bounded as in the statement of the lemma,
in view of~\adbr{\Ref{h0-dash-def}} and because
\eq\label{rmr-bnd}
    rm^{-r/2} \Le \frac2{e\log m} \Le \frac2e \Bmmi.
\en
\end{proof}

The previous lemmas quantify the approximation of~$I_r(v)$ by the population size~$Z_r$ at generation~$r$ of a
suitable branching process.   A similar argument can be used to approximate~$I_{r'}(v')$ for any $v' \ne v$, but
$I_r(v)$ and~$I_{r'}(v')$ are not independent; in particular, if $\cN_r(v)$ is given, and if $v' \in \cN_{r-1}(v)$,
the number of neighbours of~$v'$ is prescribed by knowing~$\cN_r(v)$.  Nevertheless, if $v' \notin \cN_r(v)$,
$I_{r'}(v')$ is typically close to a random variable $I_{r'}(v';v)$ that is conditionally independent of~$\cN_r(v)$,
given~$\cF_r(v)$,
and which can be coupled to the population size~$\tZ_{r'}$ at generation~$r'$ of a branching process 
with $\tZ_0=1$ and with offspring distribution~$\Po(m)$.  This is the substance of the following lemma.

\begin{lemma}\label{L4}
Conditional on~$\cF_r(v)$ and on $v' \notin \cN_r(v)$, let $\cI_{r'}(v';v)$ denote the set of vertices
at distance exactly~$r'$ from~$v'$ in the restriction~$G_v$ of the graph~$G$ to the vertices $[n] \setminus \cN_r(v)$.
Write $I_{r'}(v';v) := \card(\cI_{r'}(v';v))$ and $n' := n - N_r(v)$, and define $m' := mn'/n$.
Then,  if $n \ge m \max\{2,1/(\sqrt{m'} - 1)\}$,
the process $(I_i(v';v),\,i\ge0)$  can be constructed on the same probability space as a branching process
$(\tZ_i,\,i\ge0)$  that is independent of~$\cF_r(v)$, having $\tZ_0=1$ and with offspring distribution~$\Po(m)$, in such a way
that
\[
   \ex\{|\tZ_i - I_i(v';v)| \giv \cF_r(v)\} \Le (m/n)\hath(i,m;N_r(v)) \quad \mbox{for each}\ i \ge 1,
\]
where
\[
    \hath(i,m;N_r(v)) \Le  \h_0'(i,m) + 3im^{3i/2} + i m^{i-1} N_r(v) .
\]
\end{lemma}

\begin{proof}
As in the proof of Lemma~\ref{L1}, let $(V'_{ijk},\,i\ge1,j\ge1,1\le k\le n-1)$ be independent indicator random 
variables with $\pr[V'_{ijk}=1] = m/n$ for all $i,j,k$, independent also of $(V_{ijk},\,i\ge1,j\ge1,1\le k\le n-1)$.
Then $U'_{ij} := \sum_{k=1}^{n-1} V'_{ijk}$, $i,j \ge 1$, are independent random variables with distribution~$\Bi(n-1,m/n)$.
Let~$\hZ\un$ be the branching process constructed recursively by taking $\hZ\un_0 = 1$ and, for $i \ge 1$, by taking
$\hZ\un_i := \sum_{j=1}^{\hZ\un_{i-1}} U'_{ij}$; this branching process has offspring distribution~$\Bi(n-1,m/n)$,
and is independent of the process~$Z\un$. 

Conditional on~$\cF_r(v)$,
for any given $v'\notin \cN_r(v)$, a copy of the process  $(\cI_i(v';v),\,i\ge0)$
can also be constructed recursively from the~$V'_{ijk}$.   Take $I_0(v';v) = 1$ and $\cI_0(v';v) = \{v'\}$, and then, for $i \ge 1$,
set $\tU'_{ij} :=  \sum_{k=1}^{n - N_{i-1}(v';v) - N_r(v)} V'_{ijk}$,  $1 \le j \le I_{i-1}(v';v)$,
where $\cN_l(v';v) := \bigcup_{s=0}^l \cI_{\adbr{s}}(v';v)$ and $N_l(v';v) := \card(\cN_l(v';v))$.  Then~$\tU'_{ij}$
represents the number of edges joining the $j$-th vertex in~$\cI_{i-1}(v';v)$ (with any prescribed ordering)
to vertices in $[n] \setminus (\cN_{r}(v) \cup \cN_{i-1}(v';v))$; these are then linked at random (without replacement)
to the vertices of $[n] \setminus (\cN_{r}(v) \cup \cN_{i-1}(v';v))$, to specify the edges between
the $j$-th vertex in~$\cI_{i-1}(v';v)$
and the set of vertices $[n] \setminus (\cN_{r}(v) \cup \cN_{i-1}(v';v))$.  Carrying out this procedure
defines the set~$\cI_i(v';v)$.  Note that $I_i(v';v) \le \sum_{j=1}^{I_{i-1}(v';v)} \tU'_{ij} \le \sum_{j=1}^{I_{i-1}(v';v)} U'_{ij}$,
and that
thus $I_i(v';v) \le \hZ\un_i$ if $I_{i-1}(v';v) \le \hZ\un_{i-1}$;  induction starting with $\hZ_0 = I_0(v';v) = 1$ thus shows that
$I_i(v';v) \le \hZ\un_i$ for all $i \ge 0$.
For use in what follows, \adbr{conditional on~$N_r(v)$,} we define another intermediate branching process~$\hZ\und$,  with 
$I_i(v';v) \le \hZ\und_i \le \hZ\un_i$ for all~$i \ge 0$, by setting $\hZ\und_0 = 1$ and
$\hZ\und_i := \sum_{j = 1}^{\hZ\und_{i-1}} \sum_{k=1}^{n - N_r(v)-1} V'_{ijk}$, $i \ge 1$, thus having
offspring distribution~$\Bi(n-N_r(v)-1,m/n)$.
The process~$\hZ\und$ is the analogue of the process~$Z\un$
of Lemma~\ref{L1}, but with~$n$ replaced by the smaller value $n' := n - N_r(v)$, and
thus depends on the value of~$N_r(v)$.

It remains to show first that $I_i(v';v)$ and~$\hZ\und_i$ are close for each~$i$, and then to show that $\hZ\und_i$
is close to~$\tZ_i$ for each~$i$.  For the first of these,
since $m'/n' = m/n$ gives the correct edge probability,
Lemma~\ref{L2} can be applied to $\cG(n',m'/n') = \cG(n',m/n)$ to show that
\eq\label{4.1}
        \ex\{|\hZ\und_i - I_i(v';v)| \giv \cF_r(v)\} \Le (m'/n')\h_0'(i,m') \Le (m/n)\h_0'(i,m),\qquad i \ge 1,
\en
with the last inequality because $m' \le m$. It is then immediate that, for $i \ge 1$,
\eq\label{4.2}
  \begin{split}
  \ex\{|\hZ\un_i - \hZ\und_i| \giv \cF_r(v)\} &\Eq m_n^i \Blb 1 - \Bl\frac{n'-1}{n-1}\Br^i\Brb \\
       &\Le m_n^{i-1} miN_r(v)/n \Le im^iN_r(v)/n.
  \end{split}
\en
Finally, as for~\Ref{3.7}, the process~$\hZ\un$ can be coupled to a branching process~$\tZ$, having $\tZ_0=1$ and 
with offspring distribution~$\Po(m)$, in such a way that
\eq\label{4.3}
   \ex|\tZ_i - \hZ\un_i| \Le (i/n)\{2m^{1+3i/2} + m^i\} \Le (m/n)3im^{3i/2},
\en
where we have used $n \ge m/(\sqrt{m'} - 1)\ge m/(\sqrt{m} - 1)$.
Combining \Ref{4.1}--\Ref{4.3} 
gives the statement of the lemma.
\end{proof}

\begin{remarkno}\label{B-optimality} {\rm
It also follows from the equality in~\Ref{4.2} that
\[
    \ex\{|\hZ\un_i - \hZ\und_i| \giv \cF_r(v)\} \ \ge\ \frac{im^iN_r(v)}{n} - \frac{i(i-1)N_r^2(v)}{n^2}\,.
\]
In the bounds that follow, for fixed~$u$, both $i$ and~$r$ are chosen to be comparable in magnitude to $\log n/(2\log m)$,
and $\ex\{N_r(v)\} \ge cn^{1/2}$ and $\ex\{N_r^2(v)\} \le c' n$, for appropriate constants~$c$ and~$c'$,
so that $\ex\{m^{-i}|\hZ\un_i - \hZ\und_i| \}$ exceeds a multiple of $n^{-1/2}\log n$.
The other elements contributing to the bound in Lemma~\ref{L4} result in contributions
to $\ex\{m^{-i}|\hZ\un_i - \hZ\und_i| \}$ of the smaller asymptotic order~$O(n^{-1/2})$.  These errors carry through to
the bound in Theorem~\ref{B-dist}.
Thus there is a
real discrepancy of magnitude $n^{-1/2}\log n$, that results from approximating $I_{i}(v';v)$ by the value~$\tZ_i$ from the
\BGW\ process~$\tZ$ with offspring distribution~$\Po(m)$, rather than by the value~$\tZ'_i$ from a \BGW\ process~$\tZ'$ with
offspring distribution~$\Po(m')$, where $m' := m(1 - N_r(v)/n)$, which would be the approximation appropriate for the Bernoulli
random graph $\cG(n',m/n)$. Indeed, \adbr{conditional on~$N_r(v)$,}
the branching processes $\tZ$ and~$\tZ'$ can be simply coupled in such a way that $\tZ_i \ge \tZ'_i$
for all $i \ge 0$.  In view of the expression \Ref{B-Prob} for the conditional probability~$P$
that $D(v,v')$ exceeds a given value, this introduces a consistent underestimate of~$P$, which is then carried
into the approximating distribution function given in Theorem~\ref{B-dist}.  As a result, the error bound
in Theorem~\ref{B-dist} can be expected to be of the correct asymptotic order $O(n^{-1/2}\log n)$ as $\nti$.
\hfill$\Box$
}
\end{remarkno}

With  $\cN_j(v';v)$ as defined above,
the event that the graph distance $D(v;v')$ between $v$ and~$v' \notin \cN_r(v)$ is strictly greater than $r+r'+1$ is the
same as the event
that $\cI_r(v)$ and~$\cN_{r'}(v';v)$ have no edges between them.  By the independence of the edge indicators in the
Bernoulli graph, the conditional probability of this event, given $\cI_r(v)$ and~$\cN_{r'}(v';v)$, is just
\eq\label{B-Prob}
   \begin{split}
    P(&I_r(v), N_{r'}(v';v))\\[1ex]
      &\Def \pr\bigl[\mbox{there are no edges between $\cI_r(v)$ and~$\cN_{r'}(v';v)$}\ \giv
       I_r(v), N_{r'}(v';v)\bigr] \\
    &\,\Eq \Bl 1 - \frac mn \Br^{I_r(v) N_{r'}(v';v)}.
   \end{split}
\en
If $0 \le r'-r \le 1$, this probability is \adbr{expected to be}
well away from both $0$ and~$1$ for values of~$r$  such that~$\ex\{I_r(v)\}$ is
of magnitude comparable to~$\sqrt n$.  Defining~$r_n$ as in~\Ref{B-rn-def},
\ignore{
$:= \lfloor \half \log n/\log m \rfloor$, so that
\[
     m^{r_n} \Le \sqrt n \Eq \chi_n m^{r_n}, \qquad \mbox{where}\quad 1 \le \chi_n < m,
\]
}
this indicates choices of~$r$ of the form $r_n + t$, for~$t$ such that~$|t|$ is not too large.
Theorem~\ref{B-dist} is based on approximating the expectation of the probability~\Ref{B-Prob} in terms of branching process
quantities, using Lemmas \ref{L3} and~\ref{L4}.

\bigskip
\noindent{\bf Proof of Theorem~\ref{B-dist}}\ As well as requiring $u > -2r_n$,
we may assume that \adbr{$u \le \frac{\log n}{2\log m}$}, since otherwise the bound given in
Theorem~\ref{B-dist} is clearly satisfied.

As a prelude to the proof, we establish some useful facts.
For $Z$ and~$\tZ$ independent copies of the branching process with initial value~$1$ and with offspring distribution~$\Po(m)$,
we write
\[
    W_r \Def m^{-r}Z_r, \qquad \tW_r \Def m^{-r}\tZ_r \quad \mbox{and}\quad \tW^+_r \Def m^{-r}\sum_{i=0}^r \tZ_i.
\]
We note that $(W_r,\,r\ge0)$ and $(\tW_r,\,r\ge0)$ are non-negative martingales, and that,
a.s.\ as $r \to \infty$, $W_r \to W$, $\tW_r \to \tW$ and $\tW^+_r \to m\tW/(m-1)$, where $W$ and~$\tW$ are
independent copies of a random variable with expectation~$1$. 
Direct calculations, based, \adbr{as for \Ref{variance-bnd} and~\Ref{IsIr-bnd}, on Harris~(1963), Chapter~1, Theorem~5.1}, 
show that
\eq\label{W-moms-r}
   \ex\{W_r^2\} \Le \Bmmi \quad\mbox{and that}\quad \ex\{(\tW^+_r)^2\} \Le 3\Bmmi^3,
\en
and then, \adbb{as in Harris~(1963), Chapter~1, Theorem~8.1,} that 
\eq\label{W-moms-1}
   \ex\{(W_r - W)^2\} \Eq \frac{m^{-r}}{m-1}\,; \quad \ex\{W^2\} \Eq \Bmmi; \quad \ex\{(\tW^+)^2\} \Le 3\Bmmi^3.
\en
\ignore{
Furthermore, because, for $i < j$, 
\eqs
     \ex\{(W_i-W)(W_j-W)\} &=& \ex\bigl\{(W_i-W_j)\ex\{W_j-W \giv \cF^Z_j\}\bigr\} + \ex\{(W_j-W)^2\}  \\
                      &=& \ex\{(W_j-W)^2\} \Le \frac{m^{-j}}{m-1}\,,
\ens
it follows that
\eqa
 \ex\Bigl\{\Bigl(\tW_r^+ - \sum_{i=0}^r m^{-(r-i)}\tW\Bigr)^2\Bigr\} 
           &=& \ex\Bigl\{\Bigl(\sum_{i=0}^r m^{-(r-i)}(\tW_i - \tW)\Bigr)^2\Bigr\} \non\\
     &\le& \frac1{m-1}\Bigl\{2\sum_{i=0}^r \sum_{j=i+1}^r m^{-2r+i} +  \sum_{i=0}^r m^{-2r+i}  \Bigr\} \non \\
     &\le& m^{-r}\,\frac{m(m+1)}{(m-1)^2}\,. \label{W-moms-2}
\ena
}%
Note also that
\eq\label{W-moms-3}
   \ex\biggl|\sum_{i=0}^r m^{-(r-i)}\tW - \frac{m\tW}{m-1}\biggr|
          \Eq \sum_{l \ge 1} m^{-l-r} \Eq \frac{m^{-r}}{m-1}\,.
\en

Now, to approximate the \adbr{right hand side of}~\Ref{B-Prob}, with $r + r' + 1 = 2r_n + u$,
write $t := \lfloor (u-1)/2 \rfloor$ and define
\eq\label{B-Rss-def}
  t' \Def u - 1 - t; \quad r \Def r_n+t; \quad s \Def r_n + t'.
\en
 Supposing that $V' \notin \cN_r(V)$, write $\D_r := I_r(V) - Z_r$ and $\tD^+_s := N_{s}(V';V) - \tZ^+_{s}$, where
$Z$ and~$\tZ$ are the coupled processes of Lemmas \ref{L3} and~\ref{L4}, and $\tZ^+_{s} := \sum_{i=0}^{s} \tZ_i$.
Using the inequalities
\eq\label{inequalities}
   |(1-x)^a - (1-x)^b| \Le x|a-b| \quad\mbox{and}\quad |(1-x)^a - e^{-xa}| \Le 2x \quad\mbox{for all}\ 0 < x \le 1;\ a,b \ge 0,
\en
it follows that
\eqs
   \lefteqn{\Bigl|P(I_r(V), N_{s}(V';V)) - \exp\{-m{Z_r \tZ^+_s}/n\}\Bigr|}\\
   &&\Eq \Bigl|(1 - m/n)^{(Z_r + \D_r)(\tZ^+_{s} + \tD^+_{s})} - \exp\bigl\{-m{Z_r \tZ^+_s}/n\bigr\}\Bigr| \\
   &&\Le \frac mn \,\bigl(|\D_r|N_{s}(V';V) + Z_r|\tD^+_{s}|\bigr) + \frac{2m}n\,. 
\ens
Thus, since, as for Lemma~\ref{L1}, $\ex I_{i}(V';V) \le (m')^i \le m^i$ for all $i \ge 0$, and using Lemma~\ref{L4},
it follows that
\eqs
  \lefteqn{ \ex\Blb \Bigl|P(I_r(V), N_{s}(V';V)) - \exp\bigl\{-mZ_r \tZ^+_{s}/n\bigr\}\Bigr| \Giv V,V',\cF_r(V) \Brb
              I[V' \notin \cN_r]}\\
      &&\Le \frac mn \,\Bigl(|\D_r|\,\frac{m^{s+1}}{m-1} + \frac{Z_r} n\,\sum_{j=1}^{s}\{m(\h_0'(j,m) + 3jm^{3j/2})
                 + j m^j N_r(v)\}\Bigr)   + \frac{2m}n\,, 
\ens  
with~$\h_0'(s,m)$ as defined in~\Ref{h0-dash-def}.   Taking expectations and using Lemma~\ref{L3},
and noting also that
\eq\label{ADB-early-intersection}
   \pr[V' \in \cN_{r}] \Eq n^{-1}\ex N_{r} \Le n^{-1}\h_1(r,m) \Le n^{-1}\,\frac{m^{r_n+t+1}}{m-1}
      \Eq \frac{m^t}{\chi_n\sqrt n}\,\Bmmi,
\en
where~\Ref{B-chin-def} is used in the final equality, it follows that
\eqs
   \lefteqn{\Bigl| \pr[D(V;V') > 2r_n+u]
        - \ex\Bl \exp\Blb -(m/\chi_n^2)m^t W_{r}\,m^{t'} \tW^+_{s}\Brb \Br \Bigr|} \\
      &&\Le \frac{m}{\chi_n \sqrt n}\,\Bl\frac mn\, \h(r,m)\, \frac{m^{t'+1}}{m-1} \right.\\
      &&\qquad\qquad\qquad\quad\left. \mbox{}  + \frac mn\, m^{t} \sum_{j=1}^{s} \Blb \h_0'(j,m) + 3jm^{3j/2}
       + j m^{j-1}\,m^{r}\,\Bigl( \frac{m}{m-1}\Bigr)^2\Brb  \Br \\
       &&\qquad\mbox{}      + \frac{2m}n  + \frac{m^{t}}{\chi_n\sqrt n}\,\Bmmi ,  
\ens
where we have used Cauchy--Schwarz, Lemma~\ref{L1} and~\Ref{variance-bnd} to give
\[
   \ex\{Z_r N_r(v)\} \Le \sum_{i=0}^r \sqrt{\ex\{Z_r^2\} \ex\{(Z\un_i)^2\}} \Le m^{2r}\,\Bmmi^2,
\]
and where the last term follows from~\Ref{ADB-early-intersection}.
Crude estimates, using~\Ref{h0-dash-def} to bound~$\h_0'(j,m)$, \Ref{eta-bnd} to bound~$\h(r,m)$, and~\Ref{rmr-bnd}
in bounding $jm^{3j/2}$, as well as $\chi_n \ge 1$, $t = \lfloor (u-1)/2 \rfloor$, $t + t' = u - 1$, \adbb{$u \le
\log n/(2\log m)$ and $1/\log m \le m/(m-1)$,}
together yield
\eqa
  \lefteqn{
   \Bigl| \pr[D(V;V') > 2r_n+u]
        - \ex\Bl \exp\Blb -(m/\chi_n^2)m^{u-1} W_{r} \tW^+_{s}\Brb \Br \Bigr|} \non\\
          &&\qquad\qquad\Le c(m) m^{u/2}\max\{1,m^u\}\,  n^{-1/2}\log n,\phantom{XXXXX} \label{B-1st-approx}
\ena
with
\eq\label{c-bnd}
   c(m) \Def 2m(1+9m^{5/2})\Bmmi^4,
\en
for all~$u \ge -2r_n$. 

The approximation using $W_r$ and~$\tW^+_{s}$ is not as neat as that in which they are replaced by their limits
$W$ and~$m\tW/(m-1)$.
We thus still need to bound the expectation of the difference
\eq\label{B-Dnu-def}
  \D(n,u) \Def \exp\Blb -(m/\chi_n^2)m^{u-1} W_{r} \tW^+_{s}\Brb
                    - \exp\Bigl\{ - m^{u+1} \frac{ W \tW}{(m-1)\chi_n^2} \Bigr\}.
\en
\ignore{
However, doing so introduces an error that is typically of larger magnitude.  Since
\eqs
    \lefteqn{\ex\biggl|\exp\Blb -(m/\chi_n^2)m^s W_{r_n+t} \tW^+_{r_n+t'}\Brb
                    - \exp\Bigl\{ - m^{s+2} \frac{ W \tW}{(m-1)\chi_n^2} \Bigr\} \biggr|} \\
    &&\Le m^{s+1}\ex\bigl\{W_{r_n+t}|\tW^+_{r_n+t'} - m\tW/(m-1)| + |W_{r_n+t} - W|m\tW/(m-1)\bigr\} \\
    &&\Le m^{s+1}\Bigl\{\frac{m^{-(r_n+t'-1)/2}}{(\sqrt m - 1)\sqrt{m-1}} + \frac{m^{-(r_n+t')}}{m-1}
            + m^{-(r_n+t)/2}{\sqrt{m-1}}\Bmmi \Bigr\},
\ens
the extra error can be bounded by $c'(m) m^{s/4}\max\{m^{s/2},1\}n^{-1/4}$, where, after some calculation,
we can take
\eq\label{c-dash-bnd}
  c'(m) \Def 4m^{5/4}\Bmmi^{3/2}.
\en
}
Writing $w_n(u) := m^u/\chi_n^2$, and setting
\eq\label{Xl-def}
     X_l \Def \tW\sum_{i=0}^{l-1} m^{-i} + \sum_{i=l}^{s} m^{-i}\tW_{s-i}, \quad 0 \le l \le s+1,
\en
where, here and subsequently, empty sums are taken to be zero, so that
 $X_0 = \tW^+_{s}$ and $X_{s+1} = \tW \sum_{i=0}^{s}m^{-i}$, \adbr{we note} that
\eqa
  \D(n,u) &=& \exp\{ -w_n(u) W_{r} \tW^+_{s}\} - \exp\{ -w_n(u) W \tW^+_{s}\} \non \\
    &&\quad\mbox{} + \sum_{l=0}^{s} \bigl( \exp\{ -w_n(u) W X_l\} - \exp\{ -w_n(u) W X_{l+1}\}\bigr) \non\\
     &&\qquad\mbox{}          +  \exp\{ -w_n(u) W X_{s+1}\} - \exp\{ -w_n(u) W m\tW/(m-1)\}.
            \phantom{XXXX}\label{B-Dns-split}
\ena
We now use the inequality
\eq\label{ineq-1}
    |e^{-x} - e^{-x'} - e^{-y}(x' - x)| \Le \tfrac12 (x'-x)^2  + |x'-x|\,|y-x|,
\en
valid for all $x,x',y \ge 0$, to bound the various terms in~\Ref{B-Dns-split}.  Let $\cF^Z_r$ and $\cF^{\tZ}_r$
denote the $\s$-fields generated by $(Z_l,\,0\le l\le r)$ and $(\tZ_l,\,0\le l\le r)$, respectively, and recall that the
processes $Z$ and~$\tZ$ are independent.  We first  take $x = y = w_n(u)W_{r}\tW^+_{s}$ and $x' = w_n(u)W\tW^+_{s}$
in~\Ref{B-Dns-split}, to give
\eqa
  \lefteqn{ \bigl|\ex\bigl\{\exp\{ -w_n(u) W_{r} \tW^+_{s}\}\bigr\}
              - \ex\bigl\{\exp\{ -w_n(u) W \tW^+_{s}\}\bigr\}\bigr| } \non\\
    &&\Le |\ex\bigl\{\exp\{-w_n(u)\tW^+_{s}W_{r}\}(W - W_{r})\adbr{w_n(u)\tW^+_{s}}\}| 
         + \frac{w_n(u)^2}2\,\ex\bigl\{(W - W_{r})^2\bigr\}
                          \ex\{(\tW^+_{s})^2\} \non \\
    &&\Eq  \frac{w_n(u)^2}2\,\ex\bigl\{(W - W_{r})^2\bigr\} \ex\{(\tW^+_{s})^2\}, \label{B-diff-1}
\ena
since $\ex\{W - W_{r} \giv \cF^Z_{r}\vee \cF^{\tZ}_{s}\} = 0$ a.s.
Invoking \Ref{W-moms-r} and~\Ref{W-moms-1}, it thus follows that
\eq\label{ZW-1}
  \bigl|\ex\bigl\{\exp\{ -w_n(u) W_{r} \tW^+_{s}\}\bigr\}
              - \ex\bigl\{\exp\{ -w_n(u) W \tW^+_{s}\}\bigr\}\bigr| \Le c'_1(m)  m^{3u/2} n^{-1/2}.
\en

Next, defining
\[
     V_l \Def \tW_{s-l}\sum_{i=0}^{l-1} m^{-i} + \sum_{i=l}^{s} m^{-i}\tW_{s-i}, \quad 0 \le l \le s+1,
\]
we take $x = w_n(u)X_l W$, $z = w_n(u)V_l W$ and $x' = w_n(u)X_{l+1} W$  in~\Ref{B-Dns-split}, giving
\eqa
 \lefteqn{ \bigl|\ex\bigl\{\exp\{ -w_n(u) W X_l\}\bigr\} - \ex\bigl\{\exp\{ -w_n(u) W X_{l+1}\}\bigr\}\bigr|} \non\\
   &&\Le \bigl|\ex\bigl\{\exp\{-w_n(u) W V_l\}(X_l - X_{l+1})w_n(u)W\bigr\}\bigr| \non\\
   &&\quad\mbox{}    + \frac12 \ex\bigl\{\bigl((X_l - X_{l+1})^2 + 2|X_l - X_{l+1}|\,|V_l-X_l|\bigr) (w_n(u) W)^2\bigr\}.
      \label{B-diff-2}
\ena
Since, \adbr{from~\Ref{Xl-def},} $X_l-X_{l+1} = m^{-l}(\tW_{s-l} - \tW)$ and since~$V_l$ is $\cF^{\tZ}_{s-l}$-measurable,
and by the independence of $Z$ and~$\tZ$, it 
follows that $\ex\bigl\{\exp\{-w_n(u) W V_l\}(X_l - X_{l+1})w_n(u)W\bigr\} = 0$.  Then, from~\Ref{W-moms-1},
\[
    \ex\{(X_l - X_{l+1})^2\adbr{(w_n(u) W)^2}\} \Le \frac{m^{-l -s}}{m-1}\,w_n(u)^2 \Bmmi,
\]
and, since $V_l - X_l = (\tW_{s-l} - \tW)\sum_{i=0}^{l-1} m^{-i}$, it also follows from~\Ref{W-moms-1} that
\[
   \ex\{(V_l-X_l)^2\adbr{(w_n(u) W)^2}\} \Le \Bmmi^3 \frac{m^{-s+l}}{m-1}\,w_n(u)^2.
\]
Together, these bounds imply that
\[
    \ex\{|X_l - X_{l+1}|\,|V_l-X_l|\adbr{(w_n(u) W)^2}\} \Le \frac1{m-1}\,\adbb{\Bmmi^2} m^{-s} w_n(u)^2.
\]
As a result, we deduce that, for $0 \le l \le s$,
\eq\label{ZW-2}
  \bigl|\ex\bigl\{\exp\{ -w_n(u) W X_l\}\bigr\} - \ex\bigl\{\exp\{ -w_n(u) W X_{l+1}\}\bigr\}| \Le c'_2(m)n^{-1/2} m^{3u/2};
\en
there are \adbb{$s \le (3/4)\log n/\log m$} such contributions to the approximation error,
arising from the expression for~$\D(n,u)$ given in~\Ref{B-Dns-split}.

Finally, from~\Ref{W-moms-3}, we have
\eqs
 \lefteqn{\bigl|\ex\bigl\{\exp\{-w_n(u) W X_{s+1}\}\bigr\} - \ex\bigl\{\exp\{ -w_n(u) W m\tW/(m-1)\}\bigr\}\bigr|} \\
    &&\Le w_n(u) \frac{m^{-s}}{m-1} \Le c'_3(m) n^{-1/2}m^{u/2}. \phantom{XXXXXXXXXXXXXX}
\ens
Combining this with \Ref{B-Dns-split}, \Ref{ZW-1} and~\Ref{ZW-2} gives
\[
   \ex|\D(n,u)| \Le c'(m) m^{u/2}\max\{1,m^u\}n^{-1/2}\log n,
\]
which, with \Ref{B-1st-approx} and~\Ref{B-Dnu-def}, completes the proof of the theorem.
Note that $c(m)$ and~$c'(m)$ are uniformly
bounded in any interval of the form $1+\d \le m \le \D$, for $\d,\D > 0$;
$c(m)$ is as given in~\Ref{c-bnd}, and calculation shows that we could take $c'(m) = 5m^2\adbb{\Bmmi^4}$
in $n \ge 10$. \hfill$\blacksquare$

\section{The Bernoulli digraph}\label{s-digraph}
\ignore{
In the Bernoulli random digraph $\DG(n,m/n,\th)$, the distance between two vertices $v$ and~$v'$ can be computed along
paths that are directed from $v$ to~$v'$, along paths that are directed from $v'$ to~$v$, and indeed along paths
in which the directions of the edges are neglected.  In the latter case, the problem reduces to considering
distances in the model $\cG(n,m/n)$, and has already been covered, for~$m$ fixed, in Theorem~\ref{B-dist}.  Considering paths directed
only from $v$ to~$v'$, note that, when $v'$ does not belong to the out-neighbourhood $\cN\oplus_r(v)$ of~$v$,
the corresponding distance $D(v\to v')$ is greater than $2r+s$ exactly when the out-neighbourhood ring
$\cI_r\so(v)$ has no intersection with the in-neighbourhood $\cN\iplus_{r+s}(v';v)$, constructed within the digraph~$G$
restricted to the vertices $[n] \setminus \cN\oplus_r(v)$.  Thus the distribution of~$D(v\to v')$ can be treated
exactly as for the Bernoulli random graph $\cG(n,\mm/n)$, where $\mm := m(1+\th)/2$, and can again be approximated
using Theorem~\ref{B-dist}.  However, given two vertices $v$ and~$v'$, one can also consider the {\it joint\/} distribution
of $D(v \to v')$ and~$D(v' \to v)$, and not just their marginals.  The distributions of these distances are dependent,
because the in- and out-neighbourhoods of a vertex have points in common --- in particular, $\cN\so_r(v)$ and~$\cN\si_r(v)$
both contain the set~$\cN\sb_r(v)$, consisting of all points that can be reached from~$v$ by paths of length at most~$r$
made up only of bi-directional edges.
Despite this added complication, the strategy of proof used for the Bernoulli random graph is broadly effective here, too.
}

As observed in the introduction, the distances $D(v;v')$ and $D(v';v)$ between two vertices $v$ and~$v'$ 
in the Bernoulli random digraph $\DG(n,m/n,\th)$ are in general dependent random variables.
One reason why is that the in- and out-neighbourhoods $\cN\si_r(v)$ and~$\cN\so_r(v)$ both contain the neighbourhood
$\cN\sb_r(v)$, that consists of all vertices~$v'$ for which there is a path of length at most~$r$ from $v'$ to~$v$
that is made up entirely of bidirectional edges.  Thus it makes sense to distinguish this neighbourhood, and
to consider it alongside the neighbourhoods
 $\tcN\so_r(v) := \cN\so_r(v) \setminus \cN\sb_r(v)$ and $\tcN\si_r(v) := \cN\si_r(v) \setminus \cN\sb_r(v)$.
In particular, a branching approximation to the sizes of the triples of neighbourhood rings  
 $\tcI\so_r(v) := \tcN\so_r(v) \setminus \tcN\so_{r-1}(v)$, $\tcI\sb_r(v) := \tcN\sb_r(v) \setminus \tcN\sb_{r-1}(v)$
and $\cI\sb_r(v) := \cN\sb_r(v) \setminus \cN\sb_{r-1}(v)$ becomes appropriate.
We denote the sizes of the neighbourhoods and neighbourhood rings by
$\tN\so_r(v) := \card(\tcN\so_r(v))$, $\tN\si_r(v) := \card(\tcN\si_r(v))$, $\tI\so_r(v) = \card(\tcI\so_r(v))$,
$\tI\si_r(v) = \card(\tcI\si_r(v))$ and $I\sb_r(v) := \card(\cI\sb_r(v))$; we also
define $\cN_r(v) := \cN\so_r(v) \cup \cN\si_r(v) \cup \cN\sb_r(v)$ and $N_r(v) := \card(\cN_r(v))$.
Initially, we have $\cN\sb_0(v) = \tcI\sb_0(v) = \{v\}$
and $\tcI\so_0(v) = \tcI\si_0(v) = \emptyset$, so that $(\tI\so_0(v),\tI\si_0(v),I\sb_0(v)) = (0,0,1)$.  

Mimicking the argument for $\cG(n,m/n)$,
the first step would be to find a $3$-type branching process that bounds the numbers $(\tI\so_r(v),\tI\si_r(v),I\sb_r(v))$ 
from above.  However, the neighbourhood rings $\tcI_r\so(v)$ and~$\tcI_r\si(v)$ may have non-empty intersection.
In particular, a vertex in~$\tcI_{r-1}\so(v)$ may be linked by an out-edge to a vertex~$v' \in \tcI_{r-1}\si(v)$,
\adbr{contributing} one to each of $\tI_r\so(v)$ and~$\tI_r\si(v)$, after which the development of $\tcI_s\so(v)$, 
$s > r$, would be
linked to that of the neighbourhoods~$\tcI_{s'}\si(v)$, $s' < r$, rendering a neat branching approximation infeasible.
A similar problem arises if a vertex in~$\tcI_{r-1}\so(v)$ is linked by an out-edge to a vertex~$v' \in [n]\setminus \cN_{r-1}(v)$,
and if a vertex in~$\tcI_{r-1}\si(v)$ is linked by an in-edge to the same vertex~$v'$.
\ignore{
Each out-edge from such a~$v'$ to an as yet unexamined vertex would require both an $o$-child and an $i$-child
as offspring in the bounding branching process, introducing a further type of individual into the 
bounding branching process.
Since there is no natural definition within a branching process that would correspond to 
the offspring of two different individuals being the same, such considerations lead to difficulties.
However, such a possibility would have to be
allowed for in the bounding branching process, since
a bi-directional edge from $v' \in \tcN_r\so(v) \cap \tcN_r\si(v)$ to $[n] \setminus(\cN_r\so(v) \cup \cN_r\si(v))$
gives rise to a new element in both $\cN_r\so(v)$ and~$\cN_r\si(v)$, such a vertex~$v'$ would give rise to
an offspring distribution in the branching process that was different from those of individuals corresponding to vertices 
in $\cN_r\so(v)$ and~$\cN_r\si(v)$, introducing
a new type into the branching process. 
}

Instead, 
we consider another triple $(\hI\so_r(v),\hI\si_r(v),I\sb_r(v))$. 
We define
$\hI\so_r(v) := \card(\hcI\so_r(v))$, where $\hcI\so_r(v) \subset \tcI\so_r(v)$ consists of those vertices 
in~$[n]\setminus \cN_{r-1}(v)$ that are reached by an out- or a bidirectional edge from a vertex in~$\hcI\so_{r-1}(v)$,
or by an out-edge from a vertex in~$\cI\sb_{r-1}(v)$,
but are not reached by any other edge, whatever its direction, from $\cN_{r-1}(v)$. 
The quantities $\hcI\si_r(v)$ and~$\hI\si_r(v)$ are defined analogously.  
\adbr{Note that thus the sets $\hcI\so_r(v)$, $\hcI\si_r(v)$ and~$\cI\sb_{r}(v)$ are disjoint.}
We find a $3$-type branching process to dominate the triples $\bigl((\hI\so_r(v),\hI\si_r(v),I\sb_r(v)),\, r\ge1\bigr)$, and then
show in Lemma~\ref{DGL-new} that these triples are not too far from the triples
$\bigl((\tI\so_r(v),\tI\si_r(v),I\sb_r(v)),\, r \ge 1\bigr)$.

Write 
\eq\label{mo-def}
   m_0 \Def \half m(1+\th) \quad\mbox{and}\quad m_1 \Def \half m(1-\th).
\en
Then a vertex~$v' \in [n]$ has numbers of 
immediate neighbours $(\tI\so_1(v'),\tI\si_1(v'),I\sb_1(v'))$ distributed according to the multinomial distribution
$\MN_0(n-1;m_1/n,m_1/n,m\th/n)$ that puts the remaining probability $1 - m/n$ at each of the $n-1$ trials on the outcome $(0,0,0)$.
Conditionally on~$\cN_r(v)$,  any vertex $v' \in \hcI\so_r(v)$ has number of out-neighbours 
$\hI\sos_{r+1}(v,v')$ in 
$[n] \setminus \cN_r(v)$ distributed as $\Bi(n - N_r(v),m_0/n)$, and analogously for~$v' \in \hcI\si_r(v)$,
with $v' \in \cI\sb_r(v)$ having neighbours of the three types in $[n] \setminus \cN_r(v)$ distributed as
$\MN_0(n - N_r(v);m_1/n,m_1/n,m\th/n)$.  Note that
$\hI\so_{r+1}(v) \le \sum_{v' \in \hcI\so_r(v)}\hI\sos_{r+1}(v,v')$, sometimes without equality, because
some of the vertices in~$\hcI\sos_{r+1}(v,v')$ may also be linked to vertices in~$\cN_r(v)$ other than~$v'$.
These considerations motivate a joint construction along the lines of Lemma~\ref{L1}.

\begin{lemma}\label{DGL1}
 It is possible to construct, on the same probability space, copies of the process
 $\hI(v) := \bigl((\hI\so_r(v),\hI\si_r(v),I\sb_r(v)),\,r\ge0\bigr)$ and of a
$3$-type branching process $Z\un := \bigl((Z\un_{1,r},Z\un_{2,r},Z\un_{3,r}),\,r\ge0\bigr)$, 
with $Z\un_0 = (0,0,1)$ and with offspring distributions
\eqs
   &&(\Bi(n - 1,m_0/n),0,0) \quad \mbox{for type}\ 1\ \mbox{individuals};\\
   &&(0,\Bi(n - 1,m_0/n),0)  \quad \mbox{for type}\ 2\ \mbox{individuals}; \\
   &&\MN_0(n-1;m_1/n,m_1/n,m\th/n) \quad \mbox{for type}\ 3\ \mbox{individuals},
\ens
in such a way that $Z\un_r \ge \hI_r(v)$ (componentwise) almost surely for all $r \ge 0$.
\end{lemma}

\begin{proof}
Much as for Lemma~\ref{L1},
let $(V^l_{ijk},\,i,j,k \ge 1,\, l=1,2)$ be independent indicator random variables with
$\pr[V^l_{ijk}=1] = m_0/n$ for all $i,j,k,l$; write $\bV^l_{ijk} := V^l_{ijk}\be\ul$.  Let $(\bV^3_{ijk}\,i\ge1,j\ge1,1\le k\le n-1)$
be independent Bernoulli random $3$-vectors with distribution $\MN_0(1;m_1/n,m_1/n,m\th/n)$,
independent also of the $V^l_{ijk}$, $l=1,2$.
Then $U^l_{ij} := \sum_{k=1}^{n-1} V^l_{ijk}$, $i,j \ge 1$, $l=1,2$, are independent random variables with 
distribution~$\Bi(n-1,m_0/n)$, and $\bU^3_{ij} := \sum_{k=1}^{n-1} \bV^3_{ijk}$, $i,j \ge 1$, are independent random
$3$-vectors with distribution $\MN_0(n-1;m_1/n,m_1/n,m\th/n)$, independent also of the $U^l_{ij}$, $l=1,2$.
Define $\bU^l_{ij} := U^l_{ij}\be\ul$, $l=1,2$.

Let~$Z\un$ be the $3$-type branching process constructed recursively by taking $Z\un_0 = (0,0,1)$ and, for $r \ge 1$, by taking
$Z\un_{r} := \sum_{l=1}^3\sum_{j=1}^{Z\un_{r-1,l}} \bU^{l}_{rj}$; 
this branching process has the offspring distribution specified in the lemma.
A copy of the process~$\hI$
can also be constructed recursively from the~$\bV^l_{ijk}$, $1 \le l \le 3$.  First, take
$\hI_0(v) = (0,0,1)$ and~$\hcI_0(v) = (\emptyset,\emptyset,\{v\})$. 
Then, for $r \ge 1$ and for $1 \le j \le \hI\so_{r-1}(v)$,
take $\tU_{rj}^1 :=  \sum_{k=1}^{n - N_{r-1}(v)} V_{rjk}^1$ 
 to represent the number of out- or bidirectional edges joining the $j$-th vertex in~$\hI\so_{r-1}(v)$ (with any prescribed ordering)
to vertices outside~$\cN_{r-1}(v)$; write $\btU_{rj}^1 := \tU_{rj}^1\be\ui$.
For each $1 \le j \le \hI\so_{r-1}(v)$, $\tU_{rj}^1$ vertices can now be chosen at random (without replacement) from
$[n] \setminus \cN_{r-1}(v)$, to constitute the vertices outside $\cN_{r-1}(v)$ that are joined to the $j$-th vertex 
in~$\hcI\so_{r-1}(v)$ by out- or bidirectional edges.  Let the union over~$j$ of these sets of vertices be denoted by
$\hcI_r^{(o1)}(v)$.
The same construction, now using the random variables $V_{rjk}^2$, can be used to determine the set of vertices outside
$\cN_{r-1}(v)$ that are joined to the $j$-th vertex in~$\hcI\si_{r-1}(v)$ by in- or bidirectional edges; the union of these
sets of vertices is denoted by~$\hcI_r^{(i1)}(v)$.
Finally, for each $1 \le j \le I\sb_{r-1}(v)$, take $\btU_{rj}^3 :=  \sum_{k=1}^{n - N_{r-1}(v)} \bV_{rjk}^3$ to represent the numbers of
out-edges, in-edges and bidirectional edges joining the $j$-th vertex in~$\cI\sb_{r-1}(v)$ (with any prescribed ordering)
to vertices outside~$\cN_{r-1}(v)$, distinguishing the types of the edges.  Once again, vertices from
$[n] \setminus \cN_{r-1}(v)$ can now be chosen at random (without replacement) to constitute the vertices
outside $\cN_{r-1}(v)$ that are joined to the $j$-th vertex in~$\cI\sb_{r-1}(v)$.
The unions of these vertices, distinguished according to the types of the edges joined to them,
yields sets $\hcI_r^{(i2)}(v)$, $\hcI_r^{(i2)}(v)$ and~$\cI_r\sb$.  Finally, any vertex in
$\hcI_r^{(o*)} := \hcI_r^{o1}(v) \cup \hcI_r^{o2}(v)$ or
in $\hcI_r^{(i*)} := \hcI_r^{i1}(v) \cup \hcI_r^{i2}(v)$
that is \adbr{an} end point of more than one edge is removed, yielding the sets $\hcI_r\so$ and~$\hcI_r\si$,
respectively.

By construction, since $N_r(v) \ge 1$ for all $r \ge 0$, it follows that $\btU_{rj}\ul \le \bU^{l}_{rj}$
component\-wise, for each $r$ and~$j$.
The subsequent construction of the sets $\hcI_r\so$ and~$\hcI_r\si$, in which duplicated vertices are counted only once,
and some vertices are discarded, shows that, if
$Z\un_{r-1} \ge \hI_{r-1}$, then $Z\un_{r} \ge \hI_r$ also.  Since the inequality is true for $r=0$, it is true
for all $r \ge 0$ as required, by induction.
\end{proof}

\ignore{
by taking $\tU_{rj} :=  \sum_{k=1}^{n - N_{r-1}(v)} V_{rjk}$,  for $1 \le j \le I_{r-1}(v)$,
 to represent the number of edges joining the $j$-th vertex in~$\cI_{r-1}(v)$ (with any prescribed ordering)
to vertices outside~$\cN_{r-1}(v)$.
For each $1 \le j \le I_{r-1}(v)$, $\tU_{rj}$ vertices can now be chosen at random (without replacement) from
$[n] \setminus \cN_{r-1}(v)$, to constitute the vertices outside $\cN_{r-1}(v)$ that are joined to the $j$-th vertex in~$\cI_{r-1}(v)$;
the union over~$j$ of these sets of vertices yields~$\cI_r(v)$.
Note that thus $I_r(v) \le \sum_{j=1}^{I_{r-1}(v)} \tU_{rj} \le \sum_{j=1}^{I_{r-1}(v)} U_{rj}$, and that
thus $I_r(v) \le Z\un_r$ if $I_{r-1}(v) \le Z\un_{r-1}$;  induction starting with $Z_0 = I_0(v) = 1$ thus shows that
$I_r(v) \le Z\un_r$ for all $r \ge 0$.
these are then linked at random (without replacement) to the vertices of 
$[n] \setminus \cN_{r-1}(v)$,
to define~$\cN_r(v)$.
}

We now show that the processes $Z\un$ and~$\hI$ are reasonably close to one another.
The proof of the following lemma parallels that of Lemma~\ref{L2} for the Bernoulli random
graph.
In the remainder of the section, the designation `suitable constant' is used for a constant
depending on the values of $m$ and~$\th$, that
 can be taken to be uniform over all $m,\th$ such that $m_0 \in [1+\d,\D]$,
for any $\d > 0$ and $\D < \infty$.

\begin{lemma}\label{DGL2}
 The processes $Z\un$ and $\hI(v)$, constructed together as above, satisfy
 \[
    \ex\|Z_r\un - \hI_r(v)\|_1 \Le n^{-1}m_0^{2r}\,c_{\ref{DGL2}}(m,\th), \qquad r \ge 0,
 \]
for a suitable constant $c_{\ref{DGL2}}(m,\th)$.
\end{lemma}

\begin{proof}
 Let~$\cF_r(v)$ denote the $\s$-field generated by  $\bigl((\tcI\so_s(v),\tcI\si_s(v),I\sb_s(v)),\,0 \le s \le r\bigr)$;
 note in particular that $\bigl((\hcI\so_s(v),\hcI\si_s(v),I\sb_s(v)),\,0 \le s \le r\bigr)$ is $\cF_r(v)$-measurable.  Fix any $r \ge 0$,
 and suppress the initial vertex~$v$ in the notation.
 Let~$\cJ\so_{r+1}(v)$ consist of those vertices in $[n]\setminus \cN_r(v)$ that are reached from a vertex in~$\cI_r\sb(v)$ by an out-edge,
 or reached from a vertex in~$\hcI\so_r(v)$ by an out- or bidirectional edge.  Then, conditional on~$\cF_r(v)$, the probability that
 a given vertex~$v' \in [n]\setminus \cN_r(v)$ belongs to~$\cJ\so_{r+1}(v)$ is
 \[
      1 - (1- m_1/n)^{I\sb_r(v)}(1-m_0/n)^{\hI\so_r(v)},
 \]
where $m_0$ and~$m_1$ are defined in~\Ref{mo-def}.  Thus,  by the first and second Bonferroni inequalities,
the conditional expectation of $J\so_{r+1}(v) := \card(\cJ\so_{r+1}(v))$, given~$\cF_r(v)$, satisfies
\eqa\label{DG-J-bnd}
    \lefteqn{(n - N_r(v))f_n(I\sb_r(v), \hI\so_r(v)) ( 1 - \half f_n(I\sb_r(v), \hI\so_r(v)))}\non\\
           &&\Le \ex\{J\so_{r+1}(v) \giv \cF_r(v)\} \Le (n-1)f_n(I\sb_r(v), \hI\so_r(v)),
\ena
where $f_n(i_1,i_2) := n^{-1}\{m_1 i_1 + m_0 i_2\}$.  The vertices in~$\hcI\so_{r+1}(v)$ are those in~$\cJ\so_{r+1}(v)$
that are not reached by two or more edges, whatever their directions, from $\cN_{r}(v)$, implying that
\eq\label{DG-J-to-I}
     J\so_{r+1}(v) \ \ge\ \hI\so_{r+1}(v) \ \ge\ J\so_{r+1}(v) - K_{r+1}(v),
\en
where~$K_{r+1}(v)$ is the number of vertices in $[n]\setminus \cN_r(v)$ that are joined to more than one vertex in~$\cN_r(v)$ by 
edges of any direction.  Note that $\ex\{K_{r+1}(v) \giv \cF_r(v)\}$ is easily bounded:
\eq\label{DG-K-bnd}
    \ex\{K_{r+1}(v) \giv \cF_r(v)\} \Le n\,\binom{N_r(v)}2 (m/n)^2.
\en
Combining \Ref{DG-J-bnd}--\Ref{DG-K-bnd} and taking expectations, and writing $n_{r,1} := \ex\{\hI\so_r(v)\}$, 
$n_{r,2} := \ex\{\hI\si_r(v)\}$ and $n_{r,3} := \ex\{I\sb_r(v)\}$, it follows that
\eq\label{DG-expec-o}
   m\un_1 n_{r,3} + m\un_0 n_{r,1} - \e_{r1}  \Le  n_{r+1,1} \Le m\un_1 n_{r,3} + m\un_0 n_{r,1},
\en
where $m_l\un := (n-1)m_l/n$, $l=1,2$, and where
\eq\label{DG-eps1-bnd}
    \e_{r,1} \Def \ex\{N_r(v) f_n(I\sb_r(v), \hI\so_r(v))\} + \half n \ex\{(f_n(I\sb_r(v), \hI\so_r(v)))^2\} 
                 + \frac{m^2}{2n}\ex\{(N_r(v))^2\}.
\en
The same argument, with the natural changes of indices, applies also to bound $n_{r+1,2}$.  The inequality~\Ref{2.1} for the Bernoulli
random graph, now for~$\cG(n,m\th/n)$, gives
\eq\label{DG-expec-b}
   \th m_n n_{r,3} - \e_{r,3} \Le n_{r+1,3} \Le \th m_n n_{r,3},
\en
where
\eq\label{DG-eps3-bnd}
   \e_{r,3} \Def (m\th/n)\bigl(\ex\{I\sb_r(v) N\sb_r(v)\} + \half m\th\ex\{(I\sb_r(v))^2\}\bigr). 
\en
Defining $\bn_r := (n_{r,1},n_{r,2},n_{r,3})\TT$ and $\beps_r := (\e_{r,1},\e_{r,2},\e_{r,3})\TT$,
the statements \Ref{DG-expec-o}--\Ref{DG-eps3-bnd} provide the analogue of~\Ref{2.1} in the Bernoulli case,
in the form
\eq\label{DG-2.1}
   M_n\TT \bn_r - \beps_r \Le \bn_{r+1} \Le M_n\TT \bn_r,
\en
where  \adbr{$M_n := (n-1)M/n$}, and, writing $\g := m\th/m_0$,
\eq\label{M-matrix-def}
    M \Def m_0\begin{pmatrix}
                1 & 0 & 0 \cr
                0 & 1 & 0 \cr
                1 - \g & 1 - \g & \g
              \end{pmatrix}\,.
\en
The matrix~$M_n$ is the offspring mean matrix of the branching process~$Z\un$, and
\eq\label{DG-Zn-mean}
    (\be\uh)\TT M_n^r  \Eq \ex\{(Z_r\un)\TT\}.
\en
Iterating the inequalities~\Ref{DG-2.1}, with $n_0 = \be\uh$, then yields
\eq\label{DG-2.2}
   (\be\uh)\TT M_n^r - \sum_{s=0}^{r-1} \beps_s \TT M_n^{r-1-s} \Le \bn_r\TT \Le (\be\uh)\TT M_n^r ,
\en
the analogue of~\Ref{2.2}.

The sum in~\Ref{DG-2.2} is simplified by observing that $0 \le \g < 1$, because $0 \le \th < 1$, and hence that,
for each $r \ge 0$,
\eq\label{DG-M-powers}
    M^r \Eq m_0^r\begin{pmatrix}
                1 & 0 & 0 \cr
                0 & 1 & 0 \cr
                1 - \g^r & 1 - \g^r & \g^r
              \end{pmatrix}
         \Le m_0^r\begin{pmatrix}
                1 & 0 & 0 \cr
                0 & 1 & 0  \cr
                1 & 1 & 1
              \end{pmatrix},
\en
where the inequality \adbb{is} to be interpreted elementwise.
The quantities~$\beps_s$, whose components are given in \Ref{DG-eps1-bnd} and~\Ref{DG-eps3-bnd}, are bounded
using crude estimates.  For instance, we have
\eq\label{ADB-fII}
  f_n(I\sb_s, \hI\so_s) \Eq n^{-1}\{m_1 I\sb_s + m_0 \hI\so_s\} \Le (m_0/n)I\so_s,
\en
and~$\ex\{(I\so_s(v))^2\}$ can be bounded using \Ref{IsIr-bnd} and~\Ref{eta1-bnd} applied to the Bernoulli graph $\cG(n,m_0/n)$, 
giving
\eq\label{DG-Io2-bnd}
      \ex\{(I\so_s(v))^2\} \Le m_0^s \h_1(s,m_0) \Le m_0^{2s}\Bmmio.
\en
Then $N_s(v) \le N\so_s(v) + N\si_s(v)$, where~$\ex\{(N\so_s(v))^2\}$ can be bounded using~\Ref{N-squared-bnd} applied to
the Bernoulli graph~$\cG(n,m_0/n)$, to give
\eq\label{DG-N2-bnd}
      \ex\{(N_s(v))^2\} \Le 4\ex\{(N\so_s(v))^2\} \Le 4\h_2(s,m_0) \Le  4m_0^{2s}\Bmmio^3.
\en
Thus, from \Ref{DG-eps1-bnd}, \adbr{\Ref{ADB-fII},} \Ref{DG-Io2-bnd} and~\Ref{DG-N2-bnd}, 
together with the Cauchy--Schwarz inequality, we have
\eqa\label{DG-eps1-bnd-2}
   \e_{s,1} &\le& \frac{m_0}n \,\sqrt{\ex\{(N_s(v))^2\} \ex\{(I\so_s(v))^2\}} + \frac{m_0^2}{2n}\, \ex\{(I\so_s(v))^2\} 
            + \frac{m^2}{2n}\,\ex\{(N_s(v))^2\} \non\\
            &\le& n^{-1}c_1(m,\th)m_0^{2s},
\ena
and the same bound holds also for~$\e_{s,2}$.%
\ignore{
For~$\e_{s,3}$, the bound in~\Ref{IN+mI-squared} for the Bernoulli random
graph~$\cG(n,m\th)$ can be used directly if $m\th > 1$, giving
\[
     \e_{s,3} \Le \frac{(m\th)^{2s+2}}n\,\Bigl(\frac{m\th}{m\th-1}\Bigr)^2.
\]
If $m\th \le 1$, 
$\e_{s,3}$ can be bounded using~\Ref{IN+mI-squared-0},
which is in turn based on \Ref{eta1-bnd} and~\Ref{N-squared-bnd-0}, giving $\e_{s,3} = O(n^{-1}s^2)$ for $m\th=1$
and $\e_{s,3} = O(n^{-1})$ for $m\th < 1$.  In all cases, $\e_{s,3}$ can be bounded by $n^{-1}c_1'(m,\th)m_0^{2s}$,
for a suitable choice of~$c_1'(m,\th)$.
}
For~$\e_{s,3}$, we can use the bound~\Ref{IN+mI-squared-0}  for the Bernoulli random graph, but now
in the context of~$\cG(n,m\th)$, so that~$m$ is replaced by~$m\th$.  Since $m\th < m_0$, and since the bound~\Ref{IN+mI-squared-0}
is increasing in~$m$, we can take the crude estimate given by the bound in~\Ref{IN+mI-squared} with $m_0$ in place of~$m\th$,
giving
\[
    \e_{s,3} \Le \frac{m\th}n\,m_0^{2s+1}\,\Bmmio^2.
\]
Combining this with~\Ref{DG-eps1-bnd-2}, and substituting the result into~\Ref{DG-2.2},
it follows that
\[
    0 \Le \ex\{Z_r\un\} - \ex\{\hI_r(v)\} \Le n^{-1}m_0^{2r}\,c_2(m,\th) \bone, \qquad r \ge 0,
 \]
for a suitable constant $c_2(m,\th)$.  In view of Lemma~\ref{DGL1},
the statement of the lemma now follows.
\end{proof}

\begin{remarkno}\label{DGR1}{\rm
  Note that, for fixed~$\th$, the quantity $c_2(m,\th)$ is uniformly bounded for $m_0 \in [1+\d,\D]$,
for any $\d > 0$ and $\D < \infty$, where $m_0$ is as in~\Ref{mo-def}.} \hfill$\Box$
\end{remarkno}

As for the Bernoulli random graph, it is more satisfactory to use a single branching process for the approximation
of~$\hI(v)$. We define~$Z$ to be the three type branching process, with initial state $(0,0,1)\TT$, that has Poisson
distributed offspring numbers with mean matrix~$M$, where the numbers of offspring of type~$3$ individuals are independent
\adbr{of one another}.
\ignore{
\eqs
   &&(m_0,0,0) \quad \mbox{for type}\ 1\ \mbox{individuals};\\
   &&(0,m_0,0)  \quad \mbox{for type}\ 2\ \mbox{individuals}; \\
   &&(m_1,m_1,m\th) \quad \mbox{for type}\ 3\ \mbox{individuals},
\ens
with type~$3$ individuals having independent numbers of offspring of the different types.
}
Then, as in Lemma~\ref{L3},
we can show that the copies of the processes $\hI(v)$ and~$Z$ can also be coupled in such a way as to be close.

\begin{lemma}\label{DGL3}
Let
\eq\label{DG-n0-def}
      n_0 \Def n_0(m,\th) \Def m \max\biggl\{2,\frac1{\sqrt{m_0} - 1}, \frac1{1 - 1/\sqrt{1+\th}} \biggr\}.
\en
Then, for any $n \ge n_0$,
the processes $\hI(v)$ and~$Z$ can be constructed on the same probability space in such a way
that
\[
   \ex\|Z_r - \hI_r(v)\|_1 \Le n^{-1}m_0^{2r}\,c_{\ref{DGL3}}(m,\th),\qquad r \ge 0,
\]
where $c_{\ref{DGL3}}(m,\th)$ is uniformly bounded for $m_0 \in [1+\d,\D]$,
for any $\d > 0$ and $\D < \infty$.
\ignore{
where $\h(0,m) = 0$ and
\eq\label{eta-bnd}
    \h(r,m) \Le  \frac{m^{2r+2}}{(m-1)^3} + 7r m_0^{3r/2} \Le 4m^{2r}\,\Bmmi^3, 
               \qquad r\ge1,
\en
is increasing in both $r$ and~$m$ in $m > 1$.}
\end{lemma}

\begin{proof}  As in the proof of Lemma~\ref{L3}, we start by finding a branching process~$\tZ\un$ with Poisson offspring distributions
 that can be coupled with~$Z\un$ in such a way that $\tZ\un_r \ge Z\un_r$ componentwise, for all $r \ge 0$.  First, we define
 $\l_{0,n}$ so that $e^{-\l_{0,n}} = 1 - m_0/n$, and thus, if $n \ge n_0 \ge 2m_0$, so that
 \eq\label{DG3.1}
     m_0/n \Le \l_{0,n} \Eq -\log(1-m_0/n) \Le (m_0/n)(1 + m_0/n).
\en
Then note that the distribution $\Be(m_0/n)$ is stochastically smaller than the
 distribution~$\Po(\l_{0,n})$, and thus that $\Bi(n-1,m_0/n)$ is stochastically smaller than~$\Po(\tm_{0,n})$,
 where $\tm_{0,n} := (n-1)\l_{0,n} \le m_0\{(n-1)/n\}(1 + m_0/n)$.
Similarly, choose $\L_n$ so that $e^{-\L_n} = 1 - m/n$, \adbr{and
define
\eq\label{DG-lambda-defs}
   \l_{1,n} \Def m_1\L_n/m;\qquad \l_{2,n} \Def m\th\L_n/m.
\en
Note that, by the mean value theorem,
\[
     - \frac1y\,\log(1-y) \Le 1 + \frac{y(1+y)}{2(1-y)} \Le 1 + \frac{y(1+\th)}2\,,  \qquad \adbb{0 < y \le \th/(2 +\th);}
\]
thus, for $n \ge n_0 \ge m/\{1 - (1+\th)^{-1/2}\} \ge m(2+\th)/\th$, it follows that
\eq\label{DG3.1a}
     m/n \Le \L_n \Eq -\log(1-m/n)  \Le (m/n)(1 + m_0/n).
\en
It is thus} immediate that the distribution $\MN_0(1;m_1/n,m_1/n,m\th/n)$ is stochastically smaller than the
product distribution $\Po(\l_{1,n}) \times \Po(\l_{1,n}) \times \Po(\l_{2,n})$, and hence that the distribution
$\MN_0(n-1;m_1/n,m_1/n,m\th/n)$ is stochastically smaller than
the product distribution $\Po(\tm_{1,n}) \times \Po(\tm_{1,n}) \times \Po(\tm_{2,n})$,
where, for $n \ge n_0$,
\eqs
  \tm_{1,n} &:=& (n-1)\l_{1,n} \Le m_1\{(n-1)/n\}(1 + m_0/n);\\
  \tm_{2,n} &:=& (n-1)\l_{2,n} \Le m\th\{(n-1)/n\}(1 + m_0/n).
\ens
Hence, defining~$\tZ\un$ to be the branching process, with initial state~$\be\uh$, having
Poisson offspring distributions with mean matrix $(1 + m_0/n)M_n$,
\ignore{
\[
  \tM_n \Def (1+m_0/n)\begin{pmatrix}
                \tm_{0,n} & 0 & 0 \cr
                0 & \tm_{0,n} & 0 \cr
                \tm_{1,n} & \tm_{1,n} & \tm_{2,n}
              \end{pmatrix} \Le (1 + m_0/n)M_n,
\]
}
and with the type~3 offspring numbers independent \adbr{of one another},
it is immediate that, for $n \ge n_0$, copies of $Z\un$ and~$\tZ\un$ can be coupled in such a way that
$Z_r\un \le \tZ_r\un$ a.s.\ componentwise for all~$r \ge 0$.  This in turn implies that
\eqa
    0 &\le& \ex\|\tZ_r\un - Z_r\un\|_1 \Eq  \|\ex\{\tZ_r\un\} - \ex\{Z_r\un\}\|_1  \non \\
    &\le& \{(1 + m_0/n)^r - 1\}\|(\be\uh)\TT M_n^r\|_1  \Le (3r/n)m_0^{1 + 3r/2} , \label{DG-tilde-n-bnd}
\ena
using \Ref{3.69} (since $n \ge n_0 \ge m_0/(\sqrt{m_0} - 1)$)  and~\Ref{DG-M-powers} for the final inequality.
The same argument, with the same bound, can then
be used to compare $\tZ\un$ with a branching process~$\tZ^{(n,*)}$ having Poisson offspring distributions with mean matrix~$M_n$,
again with the type~3 offspring numbers independent, and then~$\tZ^{(n,*)}$ can be compared with~$Z$, with a bound
\eqa
  \bzero \Le \ex\|Z_r - \tZ_r^{(n,*)}\|_1 &=&  \|\ex\{Z_r\} - \ex\{\tZ_r^{(n,*)}\}\|_1 \non \\
    &\le& \{1 - (1 - 1/n)^r\}\|(\be\uh)\TT M^r\|_1  \Le (3r/n)m_0^r . \label{DG-tilde-n-bnd-2}
\ena
Combining the bounds in \Ref{DG-tilde-n-bnd} (twice) and~\Ref{DG-tilde-n-bnd-2} with that of Lemma~\ref{DGL2} completes the proof;
note that the quantity $c_{\ref{DGL3}}(m,\th)$ is uniformly bounded for $m_0 \in [1+\d,\D]$,
for any $\d > 0$ and $\D < \infty$, in view of Remark~\ref{DGR1}.
\end{proof}

It remains to show that the counts~$\hI_r\so(v)$ and~$\hI_r\si(v)$ are close to the counts $\tI_r\so(v)$ and~$\tI_r\si(v)$ that,
added to~$I_r\sb(v)$, make up the sizes $I_r\so(v)$ and~$I_r\si(v)$ of the out- and in-neighbourhood rings of~$v$.

\begin{lemma}\label{DGL-new}
 For each $r \ge 1$, we have
 \[
     \ex|\tI_r\so(v) - \hI_r\so(v)| \Le n^{-1}c_{\ref{DGL-new}} m_0^{2r}, 
 \]
 for a suitable constant $c_{\ref{DGL-new}} := c_{\ref{DGL-new}}(m,\th)$,
and the same bound holds for $\ex|\tI_r\si(v) - \hI_r\si(v)|$.
\end{lemma}

\begin{proof}
The difference $\tI_r\so(v) - \hI_r\so(v)$ is the number of vertices~$v'$ in~$\tI_r\so(v)$ that are reached from~$v$ by an out-path
$v,v_1,\ldots,v_{r-1},v'$ containing at least one vertex~$v_s$ that is joined to a vertex~$w \in \cN_{s-1}(v)$ other than~$v_s$
by an edge of any direction.  The number~$K_s$ of vertices in $[n] \setminus \cN_{s-1}(v)$ that are joined to more than
one vertex in~$\cN_{s-1}(v)$ has expectation that can be bounded using \Ref{DG-K-bnd} and~\Ref{DG-N2-bnd}, and the expected number of
out-paths of length~$r-s$ from any such vertex is~$m_0^{r-s}$.  Hence
\eqs
    \ex|\tI_r\so(v) - \hI_r\so(v)| &\le& \sum_{s=1}^{r} m_0^{r-s} \ex\{K_s(v)\} \Le 2n^{-1}m^2 m_0^{2r} \Bmmio^4,
\ens
establishing the lemma.
\end{proof}

Now we show that, if $v' \notin \cN_r(v)$, \adbr{then}
$\hI_{r'}(v')$ is typically close for each $r' \ge 0$ to a random vector $\hI_{r'}(v';v)$ that is conditionally
independent of~$\cN_r(v)$, given~$\cF_r(v)$,
and which can be coupled to the population size~$\tZ_{r'}$ at generation~$r'$ of a $3$-type branching process~$\tZ$,
having the same distribution as~$Z$, that is independent of~$\cF_r(v)$.
We do so using an argument which is essentially that of Lemma~\ref{L4}.

\begin{lemma}\label{DGL4}
Conditional on~$\cF_r(v)$ and on $v' \notin \cN_r(v)$, let $\hcI_{r'}(v';v)$ denote the sets of vertices
$(\hcI\so_{r'}(v';v),\hcI\si_{r'}(v';v),\cI\sb_{r'}(v';v))$ at distance exactly~$r'$ from~$v'$,
defined in the restriction~$G_v$ of the graph~$G$ to the vertices $[n] \setminus \cN_r(v)$ in the same way as the sets
$(\hcI\so_{r'}(v),\hcI\si_{r'}(v),\cI\sb_{r'}(v))$ are defined in~$G$.
Let $\hI_{r'}(v';v) := (\hI\so_r(v';v),\hI\si_r(v';v),I\sb_r(v';v))$
denote their cardinalities; let~$\cF_{r'}(v';v) := \s\bigl(\hI_s(v';v),\,0 \le s \le r'\bigr)$.
Then,  if $n \ge n_0$, \adbr{given in~\Ref{DG-n0-def}},
the process $\hI(v';v)$  can be constructed on the same probability space as a $3$-type branching process~$\tZ$  
that is independent of~$\cF_r(v)$ and has the same distribution as~$Z$, as defined before Lemma~\ref{DGL3},
in such a way that, for $v' \notin \cN_r(v)$,
\eqs
   \lefteqn{\ex\{\|\tZ_{s} - \hI_{s}(v';v)\|_1 \giv \cF_r(v)\}I[N_r(v) \le n(m_0-1)/(2m_0)] } \\
   &&\Le (m_0/n)c_{\ref{DGL4}} m_0^{2s} +  (3sN_r(v)/n) m_0^{s} \quad \mbox{for each}\ s \ge 1,
\ens
for a suitable constant $c_{\ref{DGL4}} = c_{\ref{DGL4}}(m,\th)$.
\end{lemma}

\begin{proof}
 To start with, as for Lemma~\ref{DGL1},
let $(V\il_{ijk},\,i,j,k \ge 1,\, l=1,2)$ be independent indicator random variables with
$\pr[V\il_{ijk}=1] = m_0/n$ for all $i,j,k,l$; write $\bV\il_{ijk} := V\il_{ijk}\be\ul$.  Let $(\bV\ih_{ijk}\,i,j,k \ge 1)$
be independent Bernoulli random $3$-vectors with distribution $\MN_0(1;m_1/n,m_1/n,m\th/n)$,
independent also of the $V\il_{ijk}$, $l=1,2$.
Then $U\il_{ij} := \sum_{k=1}^{n-1} V\il_{ijk}$, $i,j \ge 1$, $l=1,2$, are independent random variables with
distribution~$\Bi(n-1,m_0/n)$, and $\bU\ih_{ij} := \sum_{k=1}^{n-1} \bV\ih_{ijk}$, $i,j \ge 1$, are independent random
$3$-vectors with distribution $\MN_0(n-1;m_1/n,m_1/n,m\th/n)$, independent also of the $U\il_{ij}$, $l=1,2$.

Given~$\cF_r(v)$, set $n' := n - N_r(v)$ and $m' := n'm/n$, and then set $m'_0 := \half m'(1+\th)$ and $m'_1 := \half m'(1-\th)$;
note that $m'/n' = m/n$, that $m'_0/n' = m_0/n$ and that $m'_1/n' = m_1/n$.
Define $U\tl_{ij} := \sum_{k=1}^{n'-1} V\il_{ijk}$, $i,j \ge 1$, $l=1,2$;
these are conditionally independent random variables, given~$\cF_r(v)$, having distribution~$\Bi(n'-1,m'_0/n')$,
and $U\tl_{ij} \le U\il_{ij}$ for all $i,j$.  Then define
$\bU\thh_{ij} := \sum_{k=1}^{n'-1} \bV\ih_{ijk}$, $i,j \ge 1$; conditionally on~$\cF_r(v)$, these are independent random
$3$-vectors with distribution $\MN_0(n'-1;m'_1/n',m'_1/n',m'\th/n')$, and $\bU\thh_{ij} \le \bU\ih_{ij}$ for all $i,j$.
Define $\bU\il_{ij} := U\il_{ij}\be\ul$ and $\bU\tl_{ij} := U\tl_{ij}\be\ul$, $l=1,2$.

Conditional on~$\cF_r(v)$, use the random elements $\bU\tl_{ij}$, $1\le l\le 3$, to construct copies of $\hI(v';v)$ and
a $3$-type branching process~$\hZ\und$ together, as in the proof of \adbr{Lemma~\ref{DGL1}}, in such a way that,
\adbr{as in Lemma~\ref{DGL2},}
\eq\label{DG4.1}
    \ex\{\|\hZ_{s}\und - \hI_{s}(v';v)\|_1 \giv \cF_r(v)\} \Le (m_0/n)(m'_0)^{2s-1}\,c_{\ref{DGL2}}(m',\th), \qquad s \ge 1.
 \en
Note that the constant $c_{\ref{DGL2}}(m',\th)$ can be taken to be uniform in $m'_0 \le m_0$ \adbr{on the event}
$\{N_r(v) \le n(m_0-1)/(2m_0)\}$, since it then follows that $m'_0 \ge 1 + \d$ with $\d = \half(m_0-1) > 0$.
Then $\hZ\und$ can be coupled with the $3$-type branching process~$\hZ\un$ constructed, 
as in the proof of \adbr{Lemma~\ref{DGL1}},
using the random elements $\bU\il_{ij} \ge \bU\tl_{ij}$, $1\le l\le 3$;  thus $\hZ_s\un \ge \hZ\und_s$ a.s.\ for each $s \ge 0$.
Hence, much as in~\Ref{4.2},
\eq\label{DG4.2}
   \ex\{\|\hZ_{s}\un - \hZ_{s}\und\|_1 \giv \cF_r(v)\} \Le \Blb 1 - \Bl\frac{n'-1}{n-1}\Br^{s} \Brb \|(\be\uh)\TT M_n^{s} \|_1
          \Le \frac{sN_r(v)}n\,3m_0^{s}, \qquad s \ge 1.
\en
Finally, as in the proof of Lemma~\ref{DGL3}, the process~$\hZ\un$ can be coupled to a $3$-type branching process~$\tZ$,
distributed as~$Z$, in such a way that
\eq\label{DG4.3}
   \ex\|\tZ_{s} - \hZ\un_{s}\|_1 \Le (9sm_0/n) m_0^{3s/2}, \qquad s\ge1.
\en
Combining \Ref{DG4.1}--\Ref{DG4.3} establishes the lemma.
\end{proof}

\begin{remarkno}\label{DGR2}{\rm
 The same bound as in Lemma~\ref{DGL-new} can be used for the 
 expected differences $\ex\{\tI_s\so(v';v) - \hI_s\so(v';v) \giv \cF_r(v)\}$,
 where~$\tI_s\so(v';v) + I_s\sb(v';v) = N_s\so(v';v)$ is the size of the $s$-out-neighbourhood ring around the vertex~$v'$
 in~$G_v$;  and similarly for $\ex\{\tI_s\so(v';v) - \hI_s\so(v';v) \giv \cF_r(v)\}$.
 \adbr{This is because $(n')^{-1}(m'_0)^{2r} = n^{-1}m_0 (m'_0)^{2r-1} \le n^{-1}m_0^{2r}$.}} \hfill$\Box$
\end{remarkno}

Define 
\eq\label{DG-rn-def}
  r_n \Def \lfloor \half \log n/\log m_0 \rfloor \quad \mbox{and} \quad \chi_n \Def m_0^{-r_n}\sqrt n, 
\en
so that, in particular, $1 \le \chi_n \le m_0$. 

For $v' \notin \cN_r(v)$,  let $\tcI_l\so(v';v)$, $\tcI_l\si(v';v)$ and~$\cI_l\sb(v';v)$ denote the neighbourhood rings of~$v'$
in~$G^v$, defined in the same way as the neighbourhood rings $\cI_l\so(v)$, $\cI_l\si(v)$ and~$I_l\sb(v)$ of~$v$ were
defined in~$G$. Write
\[
  \cN_{s}\so(v';v) \Def \bigcup_{l=1}^s \bigl(\tcI_l\so(v';v) \cup \cI_l\sb(v';v)\bigr); \quad
  \cN_s\si(v';v) \Def \bigcup_{l=1}^s \bigl(\tcI_l\si(v';v) \cup \cI_l\sb(v';v)\bigr)
\]
for the corresponding neighbourhoods of~$v'$ in the graph~$G^v$. Then
the event~$A_{u_1,u_2}$, that the out-distance $D(v;v')$ between $v$ and~$v'$ is strictly greater than $2r_n+u_1$
and that the in-distance $D(v';v)$ is strictly greater than $2r_n+u_2$, is the
same as the event that there is no out- or bidirectional edge from
 $\cI_r\so(v)$ to~$\cN_{s_1}\si(v';v)$, and that there is no in- or bidirectional edge from
 $\cI_{r}\si(v)$ to~$\cN_{s_2}\so(v';v)$, where $r+s_1 = 2r_n+u_1-1$ and $r+s_2 = 2r_n + u_2-1$.
Since the edge indicators in the Bernoulli digraph are independent, the conditional probability of~$A_{u_1,u_2}$,
given \adbr{$\cF_r(v) \vee \tcF_{(s_1\vee s_2)}$}, 
is just $(1-m_0/n)^{E\un}$, where~$E\un$ is the \adbr{total} number of {\it distinct\/} potential 
edges between $\cI_r\so(v)$ and~$\cN_{s_1}\si(v';v)$, and 
between $\cI_{r}\si(v)$ and~$\cN_{s_2}\so(v';v)$.   Note that, because of the possible overlaps between 
$\cI_r\so(v)$ and~$\cI_{r}\si(v)$, an exact expression for $E\un$
would be complicated to write down.  However, we \adbr{recall} that the sets of 
vertices $\hcI_r\so(v)$, $\hcI_{r}\si(v)$ and~$\cI_r\sb(v)$ 
are distinct. 
We thus have
\eqa
    \lefteqn{\hI_{r}\so N_{s_1}\si(v';v) + I_r\sb(v)\card\bigl(\cN_{s_1}\si(v';v) \cup \cN_{s_2}\so(v';v)\bigr)
                   + \hI_{r}\si N_{s_2}\so(v';v) \Le E\un } \non\\
    &\le& 
     \tI_{r}\so N_{s_1}\si(v';v) + I_r\sb(v)\card\bigl(N_{s_1}\si(v';v) \cup N_{s_2}\so(v';v)\bigr) + \tI_{r}\si N_{s_2}\so(v';v). 
              \label{DG-edge-ineq}
\ena
\ignore{
and, in particular, that
\eq\label{DG-E1+E2}
   \hI_{r}\so N_{s}\si(v';v) + \hI_{r'}\si N_{s'}\so(v';v) \Le E\un_1 + E\un_2;
\en
note that, since $\hcI_{r}\so(v)$ and~$\hcI_{r}\si(v)$ are disjoint, the edges being counted on the left hand side
of~\Ref{DG-E1+E2} are all distinct.
}

By Lemma~\ref{DGL3}, $\hI_r\so(v)$, $\hI_{r}\si$ and~$I_r\sb(v)$ 
can be jointly
approximated, using the $3$-type branching process~$Z$, by~$Z_{1,r}$, $Z_{2,r}$ and~$Z_{3,r}$, 
and $\hI_r\so(v)$ and~$\hI_{r}\si$ are close to $\tI_r\so$ and~$\tI_{r}\si$ respectively, using also Lemma~\ref{DGL-new}.  
\ignore{
Then $I_r\so(v) = \tI_r\so(v) + I_r\sb(v)$ and $I_{r'}\si(v) = \tI_{r'}\si(v) + I_{r'}\sb(v)$, and the quantities
$I_r\sb(v)$ and~$I_{r'}\sb(v)$ are much smaller than $\tI_r\so(v)$ and~$\tI_{r'}\si(v)$, respectively.
Observe that the processes $\baZ_1 := (Z_{1,r} + Z_{3,r},\,r\ge0)$, $\baZ_2 := (Z_{2,r} + Z_{3,r},\,r\ge0)$ and 
$(Z_{3,r},\,r\ge0)$ are (dependent)
Galton--Watson processes starting from a single individual, $\baZ\ui$ and~$\baZ\ut$ having offspring distribution~$\Po(m_0)$, 
and~$\baZ\uh$ having offspring distribution~$\Po(m\th)$.  Similarly, define the processes
$\tbaZ^+_{1,s_2} := \tZ^+_{1,s_2}+ \tZ^+_{3,s_2}$ and~$\tbaZ^+_{2,s_1} := \tZ^+_{2,s_1}+ \tZ^+_{3,s_1}$, where $\tZ^+_t := \sum_{l=1}^t \tZ_l$.
}
Then
$N_{s_2}\so(v';v)$ and~$N_{s_1}\si(v';v)$ can be jointly approximated, in terms of the $3$-type branching process~$\tZ$,
using  Lemma~\ref{DGL4} and Remark~\ref{DGR2}.

As a result of these considerations, taking $u_1 \ge u_2$ without real loss of generality, and hence $s_1 \ge s_2$,
the unconditional probability of~$A_{u_1,u_2}$ can be well approximated by
\eq\label{DG-prob-approx}
   \ex\Bigl\{\Bigl(1- \frac{m_0}n \Bigr)^{Z(r,s_1,s_2)} \Bigr\},
\en
or by
\eq\label{DG-prob-appx-def}
  P_n(r,s_1,s_2) \Def \ex\bigl\{\exp\{-n^{-1}m_0 Z(r,s_1,s_2)\}\bigr\},
\en
where
\eq\label{DG-Zrss-def}
    Z(r,s_1,s_2) \Def Z_{1,r} (\tZ^+_{2,s_1} + \tZ^+_{3,s_1}) + Z_{3,r}(\tZ^+_{2,s_1} + \tZ^+_{3,s_1} + \tZ^+_{1,s_2}) 
            + Z_{2,r} (\tZ^+_{1,s_2} + \tZ^+_{3,s_2}).
\en
The expression in~\Ref{DG-Zrss-def} 
can then be approximated using limiting random variables associated with the branching processes $Z$ and~$\tZ$,
giving a neater approximation to $P_n(r,s_1,s_2)$.
The details are contained in the following theorem, whose proof has much the same structure
as that of Theorem~\ref{B-dist}. Frequent use is again made of the inequalities~\Ref{inequalities}.

In preparation, we note some asymptotic properties of the $3$-type branching process~$Z$.
First, the processes $\baZ_1 := (Z_{1,j} + Z_{3,j},\,j\ge0)$, $\baZ_2 := (Z_{2,j} + Z_{3,j},\,j\ge0)$ and
$(Z_{3,j},\,j\ge0)$ are (dependent)
Galton--Watson processes starting from a single individual, $\baZ_1$ and~$\baZ_2$ having offspring distribution~$\Po(m_0)$, 
and~$Z_3$ having offspring distribution~$\Po(m\th)$.  Thus,  if $m_0 > 1$,
it follows that $W^*_{1,j} := m_0^{-j}\baZ_{1,j}$ and $W^*_{2,j} := m_0^{-j}\baZ_{2,j}$ 
both converge a.s.\ to (dependent) limits
\eq\label{DG-W-lims}
   W^*_1 \Def \lim_{j\to\infty} W^*_{1,j} \quad\mbox{and}\quad W^*_2 \Def \lim_{j\to\infty} W^*_{2,j},
\en
each having mean~$1$.  If $m\th > 1$, the martingale $W_{3,j} := (m\th)^{-j}Z_{3,j}$ also converges a.s.\ to a limit~$W_3$
having mean~$1$; if $m\th \le 1$, 
$W_3=0$ a.s.
Note that, writing
\eq\label{DG-Wsp-def}
    \baZp_{j,i} \Def \sum_{l=0}^i \baZ_{j,l} \quad\mbox{and}\quad
           W\usp_{j,i} \Def m_0^{-i}\baZp_{j,i} \Eq \sum_{l=0}^i m_0^{-l}W^*_{j,i-l},\quad j=1,2,
\en
it follows also that, as $i \to \infty$,
\eq\label{DG-Wsp-lim}
     W\usp_{j,i} \ \to\ \Bmmio W^*_j \text{ a.s.}, \quad j=1,2;
\en
similarly, if $m\th > 1$, we have
\eq\label{DG-W3p-lim}
   W^+_{3,i} \Def  \sum_{l=0}^i (m\th)^{-l}W_{3,i-l} \ \to\ \Bmmith W_3 \text{ a.s.}
\en

The following second moment bounds can now be derived from
 \Ref{W-moms-r} and~\Ref{W-moms-1}: 
\begin{align}
   &\ex\{(W^*_{\jj,\ii} - W^*_{\jj})^2\} = \frac{m_0^{-\ii}}{m_0-1}\,,\ \jj=1,2;
       \quad &\ex\{(W_{3,\ii} - W_3)^2\} = \frac{(m\th)^{-\ii}}{m\th-1}\,, \ m\th > 1;
                 \label{DG-W-diff-square} \\
    &\ex\{(\tW\usp_{\jj,\ii})^2\} \Le 3\Bmmio^3,\ \jj=1,2;  \quad &\ex\{(\tW^+_{3,\ii})^2\} \Le 3\Bmmith^3, \ m\th > 1. \label{DG-Wsum-square}
\ignore{
   &\ex\Bigl\{\Bigl(\tW\usp_{\jj,\ii} - \sum_{l=0}^\ii m_0^{-(\ii-l)}\tW^*_{\jj}\Bigr)^2\Bigr\}
           \Le m_0^{-\ii}\,\frac{m_0(m_0+1)}{(m_0-1)^2}\,, \quad \jj=1,2; \label{DG-Wsum-diff-square} \\
    &\ex\Bigl\{\Bigl(\tW^+_{3,\ii} - \sum_{l=0}^\ii (m\th)^{-(\ii-l)}\tW_3\Bigr)^2\Bigr\}
           \Le (m\th)^{-\ii}\,\frac{m\th(m\th+1)}{(m\th-1)^2}\,, \quad m\th > 1. \phantom{XXXXXXXXXX}\label{DG-Wsum-diff-square-2}
}
\end{align}
We also note that
\eq\label{DG-Z3-mean}
      m_0^{-j}\ex\{Z_{3,j}\} \Eq \Bigl(\frac{m\th}{m_0}\Bigr)^j. 
\en

\begin{theorem}\label{DG-dist}
For any $n \ge n_0$, as defined in~\Ref{DG-n0-def},
let  $V_n$ and~$V'_n$ be a pair of vertices chosen independently and uniformly at random from the vertex set of
a Bernoulli digraph $G \sim \DG(n,m/n,\th)$. For $u_1,u_2 > -2r_n$, define
\[
    A_{u_1,u_2}\un \Def \{D(V_n;V_n') > 2r_n + u_1\} \cap \{D(V_n';V_n) > 2r_n + u_2\},
\]
where~$r_n$ is as defined in~\Ref{DG-rn-def}.   Let $(W^*_1,W^*_2,W_3)$ and $(\tW^*_1,\tW^*_2,\tW_3)$ be independent copies of
the a.s.\ limit as $j \to \infty$ of $\bigl(m_0^{-j}(Z_{1,j}+Z_{3,j}),m_0^{-j}(Z_{2,j}+Z_{3,j}), (m\th)^{-j}Z_{3,j}\bigr)$,
where~$Z$ is a $3$-type branching process with
initial state $(0,0,1)\TT$, having Poisson offspring distributions with mean matrix~$M$ given in~\Ref{M-matrix-def},
with independent offspring numbers of the different types.  Define
\eq\label{DG-epsn-def}
     \e_n \Def \e_n(m,\th) \Def (m\th/m_0)^{2r_n} \mathbf{1}_{\{m\th > \sqrt{m_0}\}}
\en
and
\eq\label{DG-hW-def}
     \hW(u_1,u_2,\e) \Def (m_0^{u_1-1}W^*_1 \tW^*_2  + m_0^{u_2-1} W^*_2\tW^*_1 )\Bmmio
              - \e_n (m\th)^{(u_1 \wedge u_2)-1} W_3 \tW_3 \Bmmith.
\en
Then
\eqs
   \bigl|\pr[A_{u_1,u_2}\un]
   - \ex\bigl\{\exp\{- (m_0/\chi_n^2) \hW(u_1,u_2,\e_n)\bigr\}\bigr|
   &\le& C(m,\th) m_0^{u^*/2}\max\{1,m^{u^*}\}n^{-1/2}\log n  
\ens
where $u^* := \max\{u_1,u_2\}$.
\ignore{
\[
    \a \Def \begin{cases}
             1 - \g'/2  & \text{ if }  m\th \le 1; \\ 
             1 - \g'    & \text{ if } m\th > 1,    
            \end{cases}
\]
with $\g' := \log(m\th)/\log m_0$.
}%
The constant $C(m,\th)$ 
can be chosen to be uniform in $m,\th$ such that $1 + \d \le m_0 \le \D$, for any $\d,\D > 0$.
\end{theorem}

\begin{remarkno}\label{DG-thm-rk}
{\rm
The quantity $(m\th/m_0)^{2r_n}$ is of order $O(n^{-1/2})$ if $m\th \le \sqrt{m_0}$.  This explains
why the term involving the product $W_3 \tW_3$ can be dispensed with in the exponent, if $1 < m\th \le \sqrt{m_0}$;
its contribution is of no larger order than the remaining error.
If $m\th \le 1$, the product is in any case equal to zero.

 If $u_2$ is taken to be large and negative in the approximation given in Theorem~\ref{DG-dist}, the resulting
approximation to the distribution of $D(V;V')$ is the same as that for~$D(V,V')$ in the Bernoulli random graph $\cG(n,m_0/n)$,
as it should be.   If $\th = 1$, \adbr{entailing $m_0=m$ and $\e_n = 1$,}
it follows that $W^*_1 = W^*_2 = W_3$ a.s., in which case, writing $W$ for~$W_3$ and $\tW$
for~$\tW_3$, we have
\[
   (m_0/\chi_n^2) \hW(u_1,u_2,\e) \Eq m(m^{(u_1 \vee u_2)} W \tW)/\{(m-1)\chi_n^2\}.
\]
This implies that the limiting distribution of inter-point distances is concentrated on the diagonal $u_1 = u_2$,
and, since the Bienaym\'e--Galton--Watson process~$Z_{3,\cdot}$ has offspring distribution~$\Po(m)$ when $\th=1$,
the marginal distribution is the same as that for the model $\cG(n,m/n)$, again as it should be.

For $Z_0 := (0,0,1)\TT$, we have 
$\var W^*_1 = \var W^*_2 = 1/(m_0-1)$, from~\Ref{W-moms-1}.  Then, by the independence of lines of descent in the
branching process,
\[
    \law\bigl((W^*_1,W^*_2) \giv Z_1\bigr)
      \Eq \law \Bigl(m_0^{-1}\Bigl\{\sum_{l=1}^{Z_{11}} (W^{[1]}_{1,l},W^{[1]}_{2,l})
                  + \sum_{l=1}^{Z_{12}} (W^{[2]}_{1,l},W^{[2]}_{2,l}) + \sum_{l=1}^{Z_{13}} (W^{[3]}_{1,l},W^{[3]}_{2,l})
                         \Bigr\}\Bigr),
\]
where, for each~$k,l$, $\law\bigl((W^{[k]}_{1,l},W^{[k]}_{2,l})\bigr) = \law\bigl((W^*_1,W^*_2) \giv Z_0=\be^{(k)}\bigr)$,
and the pairs $(W^{[k]}_{1,l},W^{[k]}_{2,l})$ are mutually independent.
Hence
\eqs
   C_{12} &:=& \cov(W^*_1,W^*_2) \Eq \ex\bigl\{\cov(W^*_1,W^*_2 \giv Z_1)\bigr\}
             + \cov\bigl(\ex\{W_1^*\giv Z_1\},\ex\{W_1^*\giv Z_1\}\bigr) \\
       &=& \ex\{Z_{13}m_0^{-2} C_{12}\} + m_0^{-2}\cov(Z_{11} + Z_{13}, Z_{12} + Z_{13})
       \Eq m\th(C_{12} + 1)/m_0^2,
\ens
from which it follows that
\[
    \corr(W^*_1,W^*_2) \Eq \frac{\g(m_0-1)}{m_0 - \g}\,,\quad\mbox{where}\quad 0 \le \g := m\th/m_0 \le 1.
\]
Analogously, if $m\th > 1$, we have $\var W_3 = 1/(m\th-1)$ and
\[
  \hskip1in \corr(W^*_1,W_3) \Eq \corr(W^*_2,W_3) \Eq \sqrt{\frac{m\th - 1}{m_0-1}}\,. \hskip1.5in \Box
\]
}
\end{remarkno}

\begin{proof}
Suppressing the index~$n$ where possible, let $\tcF_r := \s(V,V')\vee \cF_r(V)$.
We begin by taking $u_1 \ge u_2 > - 2r_n$, assuming that $u_1 \le \half\log n/\log m_0$, since the inequality in the theorem is
immediately true otherwise, with $C(m,\th) = 1$.
For the event $A_{u_1,u_2}$ that $D(V;V') > 2r_n+u_1$ and $D(V';V) > 2r_n+u_2$,
define
\eq\label{ADB-A2}
  t \Def \lfloor (u_1-1)/2\rfloor, \quad t_1 \Def u_1 - t - 1 \quad\mbox{and}\quad 
  t_2 \Def u_2 - t - 1,
\en
noting that then $t_1 \ge t_2$; 
write
\eq\label{DG-rss-def}
  r \Def r_n + t; \quad s_1 \Def r_n + t_1 ; \quad s_2 \Def  r_n + t_2.
\en

In order to reach the approximation given in~\Ref{DG-prob-approx} for the probability $\pr[A_{u_1,u_2}]$,
let $N_{s_1,s_2}(V';V) := \card(N_{s_1}\si(V';V) \cup N_{s_2}\so(V';V))$, and set $B_r\ui := \{V' \notin \cN_r(V)\}$.
We first observe, from \adbr{\Ref{inequalities} and}~\Ref{DG-edge-ineq}, that, on the event~$B_r\ui$, we have
\eqa
0 &\le& \Bigl(1 - \frac{m_0}n\Bigr)^{ \hI_{r}\so N_{s_1}\si(V';V) +  I_r\sb(V)N_{s_1,s_2}(V';V)
                       + \hI_{r}\si N_{s_2}\so(V';V) }
                          - \Bigl(1 - \frac{m_0}n\Bigr)^{E\un} \non  \\
  &\le& \frac{m_0}n\,\Bigl(\{\tI_r\so(V) - \hI_r\so(V)\} N_{s_1}\si(V';V)
                    + \{\tI_{r}\si(V) - \hI_{r}\si(V)\} N_{s_2}\so(V';V) \Bigr). \phantom{XXX}
        \label{DG-E1-E2-sandy}                      
\ena
\ignore{
\eqa
  0 &\le& \Bigl(1 - \frac{m_0}n\Bigr)^{\hI_r\so(V) N_s\si(V';V) + \hI_{r'}\si(V) N_{s_2}\so(V';V)} 
         - \Bigl(1 - \frac{m_0}n\Bigr)^{E\un_1 + E\un_2} \non  \\
    &\le& \frac{m_0}n\,\Bigl(\{\tI_r\so(V) + I_r\sb(V) - \hI_r\so(V)\} N_s\si(V';V) + \{\tI_{r'}\si(V) + I_{r'}\sb(V) - \hI_{r'}\si(V)\} N_{s_2}\so(V';V) \Bigr).
        \label{DG-E1-E2-sandy}
\ena
}%
Now, conditional on~$\tcF_r \cap B_r\ui$, with the coupling of Lemma~\ref{DGL4},
$\tI_l\si(V';V)\} \le \hZ\un_{2,l} + \hZ\un_{3,l}$ a.s., \adbr{$0 \le l \le s_1$,}
and hence 
\[
    \ex\{N_{s_1}\si(V';V)\giv \tcF_r\}I[B_r\ui] \Le \sum_{l=1}^{s_1} m_0^l \Le m_0^{s_1} \Bmmio \quad \mbox{a.s.}; 
\]
analogously, $\ex\{N_{s_2}\so(V';V) \giv \cF_r\}I[B_r\ui] \le m_0^{s_2} \Bmmio$ a.s.\ also. 
Then, from Lemma~\ref{DGL-new},
\[
    \ex\{\tI_r\so(V) - \hI_r\so(V)\} \Le n^{-1}m_0^{2r}\,c_{\Ref{DGL-new}}(m,\th),
\]
and the same bound holds for $\ex\{\tI_{r}\si(V) - \hI_{r}\si(V)\}$.  
Combining these observations, it follows that
\eqa
  0 &\le& \ex\Bigl\{\Bigl(1 - \frac{m_0}n\Bigr)^{\hI_{r}\so(V) N_{s_1}\si(V';V) +  I_r\sb(V)N_{s_1,s_2}(V';V) 
            + \hI_{r}\si(V) N_{s_2}\so(V';V)}
         - \ex\Bigl\{\Bigl(1 - \frac{m_0}n\Bigr)^{E\un}\Bigr\} \non  \\
    &\le& \pr[(B_r\ui)^c] + \frac{m_0}n\,\Bigl(n^{-1}(m_0^{2r+s_1} + m_0^{2r+s_2})\,c_{\ref{DGL-new}}(m,\th)\Bmmio\Bigr) \non\\
    &\le& \pr[(B_r\ui)^c] + n^{-2}m_0^{3r_n+u+t_1}\,c_{\ref{DG-appx-1}}(m,\th),
        \label{DG-appx-1}
\ena
for a suitable constant $c_{\ref{DG-appx-1}}(m,\th)$.

Define the processes $\tbaZ_1$ and~$\tbaZ_2$ by $\tbaZ_{1,l} := \tZ_{1,l} + \tZ_{3,l}$ and~$\tbaZ_{2,l} := \tZ_{2,l}+ \tZ_{3,l}$,
\adbr{where~$\tZ$ is the $3$-type branching process, having the same distribution as~$Z$, introduced in Lemma~\ref{DGL4};}
write $\tbaZp_{i,t} := \sum_{l=1}^t \tbaZ_{i,l}$, $i = 1,2$.
The next step is to examine 
\eqs
   N_{s_1}\si(V';V) - \tbaZp_{2,s_1} &=& \sum_{l=1}^{s_1} (\tI_l\si(V';V) + I_l\sb(V';V) - \tbaZ_{2,l}) \\
        &=& \sum_{l=1}^{s_1} \{(\tI_l\si(V';V) - \hI_l\si(V';V)) + (\hI_l\si(V';V) + I_l\sb(V';V) - \tbaZ_{2,l})\}.
\ens
From Lemma~\ref{DGL-new} and Remark~\ref{DGR2}, for $B_r := B_r\ui \cap \{N_r(V) \le n(m_0-1)/(2m_0)\}$, it follows that
\[
    \ex\{\tI_l\si(V';V) - \hI_l\si(V';V) \giv \tcF_r\} I[B_r]  \Le n^{-1}m_0^{2l}c_{\ref{DGL-new}}(m,\th) \ \mbox{a.s.}
\]
Then, from Lemma~\ref{DGL4},  we have
\[
    \ex\{\|\tZ_{l} - \hI_{l}(V';V)\|_1 \giv \tcF_r\}I[B_r] \Le n^{-1}m_0 c_{\ref{DGL4}}(m,\th) m_0^{2l} +  (3lN_r(V)/n) m_0^{l}.
\]
Hence
\eq \label{DG-NVV-bnd}
    \ex\{|N_{s_1}\si(V';V) - \tbaZp_{2,s_1}| \giv \tcF_r\}I[B_r]  \Le \h_{s_1},
\en
where
\eq\label{DG-h-def}
   \h_s \Def n^{-1}m_0^{2s_1}c_{\Ref{DG-h-def}} + 3sn^{-1}N_r(V) m_0^s \Bmmio\,,
\en
and
\eq\label{DG-h-def-2}
   c_{\ref{DG-h-def}} \Def c_{\ref{DG-h-def}}(m,\th)
            \Def m_0^2\,\frac{c_{\ref{DGL-new}}(m,\th) + m_0 c_{\ref{DGL4}}(m,\th)}{m_0^2 - 1}
\en
is a suitable constant.
The same bound, with $s_2$ for~$s_1$, holds for the expected difference
$\ex\{|N_{s_2}\so(V';V) - \tbaZp_{1,s_2}| \giv \tcF_r\}I[B_r]$.
Then we have
\[
N_{s_1,s_2}(V';V) \Eq N_{s_1}\si(V';V) + \card(\cN_{s_2}\so(V';V) \setminus \cN_{s_1}\si(V';V)),
\]
with
\[
     \cN_{s_2}\so(V';V) \setminus \cN_{s_1}\si(V';V)  \Eq \tcN_{s_2}\so(V';V) \Eq \bigcup_{l=1}^{s_2} \tcI_{l}\so(V';V).
\]
Thus it follows that
\[
    |N_{s_1,s_2}(V';V) - \tbaZp_{2,s_1} - \tZ^+_{1,s_2}| \Le |N_{s_1}\si(V';V) - \tbaZp_{2,s_1}| + \sum_{l=1}^{s_2} |\tI_{l}\so(V';V) -  \tZ_{1,l}|;
\]
\ignore{
\eqs
    \Bigl|\tN_{s'}\so(V';V) - \sum_{l=1}^{s'} \tZ_{1,l}\Bigr| &\le& \sum_{l=1}^{s'} |\tI_{l}\so(V';V) -  \tZ_{1,l}|,
\ens
}
taking expectations, using~\Ref{DG-NVV-bnd} together with
 Lemmas \ref{DGL-new} and~\ref{DGL4} and Remark~\ref{DGR2}, this gives
\eqa
    \ex\{|N_{s_1,s_2}(V';V) - \tbaZp_{2,s_1} - \tZ^+_{1,s_2}| \giv \tcF_r\}I[B_r] 
       &\le& \h_{s_1} + \h_{s_2}. \label{Nss-bnd}
\ena
Combining the results of \Ref{DG-NVV-bnd} and~\Ref{Nss-bnd}, we have shown that
\eqa
\lefteqn{\Bigl|\ex\Bigl\{\Bigl(1 - \frac{m_0}n\Bigr)^{\hI_{r}\so N_{s_1}\si(V';V) + I_r\sb(V)N_{s_1,s_2}(V';V)
            + \hI_{r}\si N_{s_2}\so(V';V)} \Giv \tcF_r\Bigr\} } \non\\
   &&\qquad\mbox{}  -   \ex\Bigl\{\Bigl(1 - \frac{m_0}n\Bigr)^{\hI_{r}\so(V) \tbaZp_{2,s_1} +   I_r\sb(V)(\tbaZp_{2,s_1} + \tZ^+_{1,s_2}) 
            + \hI_{r}\si(V) \tbaZp_{1,s_2}} \Giv \tcF_r\Bigr\} \Bigr|  I[B_r] \non\\
    &&\Le n^{-1}m_0\,\{\hI_{r}\so(V) \h_{s_1} + I_r\sb(V)(\h_{s_1} + \h_{s_2}) + \hI_{r}\si(V) \h_{s_2}\} \phantom{XXXXXXXXXXX} \non\\
    &&\Le n^{-1}m_0\,\{N_{r}\so(V) \h_{s_1} + N_r\si(V)\h_{s_2}\}.
        \label{DG-next-bnd}
\ena
Using~\Ref{N-squared-bnd} applied to the out- and in-neighbourhood processes, together with the Cauchy--Schwarz inequality,
we have   
\[
    \ex\{N_{r}\so(V) N_r(V)\} \Le 2m_0^{2r}\Bmmio^3.
\]
Taking unconditional expectations, and recalling that $s_2 \le s_1$, \Ref{DG-next-bnd} thus implies that
\eqa
\lefteqn{\biggl|\ex\Bigl\{\Bigl(1 - \frac{m_0}n\Bigr)^{\hI_{r}\so N_{s_1}\si(V';V) + I_r\sb(V)N_{s_1,s_2}(V';V)
            + \hI_{r}\si N_{s_2}\so(V';V)} \Bigr\} } \non\\
   &&\qquad\mbox{}  -   \ex\Bigl\{\Bigl(1 - \frac{m_0}n\Bigr)^{\hI_{r}\so(V) \tbaZp_{2,s_1} +   I_r\sb(V)(\tbaZp_{2,s_1} + \tZ^+_{1,s_2}) 
            + \hI_{r}\si(V) \tbaZp_{1,s_2}} \Bigr\} \biggr|  \non\\
    &&\Le \pr[(B_r\ui)^c] + \adbr{m_0^{u+t_1}}\,c_{\ref{DG-appx-2}}(m,\th) n^{-1/2}\log n,  \label{DG-appx-2}
\ena
for a suitable constant $c_{\ref{DG-appx-2}}(m,\th)$.

The next step is to replace $\hI_{r}\so(V)$, $I_r\sb(V)$ and~$\hI_{r}\si(V)$ in the exponent of the approximation~\Ref{DG-appx-2}
by the quantities $Z_{1,r}$, $Z_{3,r}$ and~$Z_{2,r}$ respectively.  This can be easily achieved, using Lemma~\ref{DGL3}
and the moments of the process~$\tbaZp$, \adbr{introduced following~\Ref{DG-appx-1},} giving
\eqa
 \lefteqn{\biggl|\ex\Bigl\{\Bigl(1 - \frac{m_0}n\Bigr)^{\hI_{r}\so(V) \tbaZp_{2,s_1} +   I_r\sb(V)(\tbaZp_{2,s_1} + \tZ^+_{1,s_2}) 
            + \hI_{r}\si(V) \tbaZp_{1,s_2}} \Bigr\} 
             - \ex\Bigl\{\Bigl(1 - \frac{m_0}n\Bigr)^{Z(r,s_1,s_2)} \Bigr\} \biggr| }\non\\
    &&\Le  n^{-2}m_0^{3r_n+u_1+t_1}\,c_{\ref{DG-appx-3}}(m,\th), \phantom{XXXXXXXXXXXXXXXXXXXX}  \label{DG-appx-3}
\ena
for a suitable constant $c_{\ref{DG-appx-3}}(m,\th)$.  Then it follows from~\Ref{inequalities} that
\eq\label{DG-appx-4}
  \biggl|\ex\Bigl\{\Bigl(1 - \frac{m_0}n\Bigr)^{Z(r,s_1,s_2)} \Bigr\} 
   -   \ex\Bigl\{\exp\{-n^{-1}m_0 Z(r,s_1,s_2)\} \Bigr\} \biggr| 
   \Le  2n^{-1}m_0.  
\en
Finally, it follows \adbr{using} \Ref{N-squared-bnd} that 
\eq\label{DG-B-bnd}
     \pr[(B_r\ui)^c]  \Eq n^{-1}\ex N_r(v) \Le n^{-1}m_0^{r_n+u}c_{\ref{DG-B-bnd}}(m,\th),
\en
for a suitable constant~$c_{\ref{DG-B-bnd}}(m,\th)$.

Collecting \Ref{DG-appx-1}, \Ref{DG-appx-2}, \Ref{DG-appx-3} and~\Ref{DG-appx-4}, and using~\Ref{DG-B-bnd}, 
it thus follows that
\eq\label{DG-main-bnd-1}
   \bigl| \pr[A_{u_1,u_2}] - \ex\bigl\{\exp\{-n^{-1}m_0 Z(r,s_1,s_2)\} \bigr\} \bigr|
      \Le C_0(m,\th)  m_0^{u/2}\max\{1,m_0^{u}\} n^{-1/2}\log n,
\en
for a suitable constant~$C_0(m,\th)$.

Replacing $n^{-1}m_0 Z(r,s_1,s_2)$ in the probability $\ex\bigl\{\exp\{-n^{-1}m_0 Z(r,s_1,s_2)\} \bigr\}$ by the expression
$(m_0/\chi_n^2) \hW(u_1,u_2,\e_n)$, which involves only the limiting random variables
$W^*_1$, $W^*_2$, $W_3$, $\tW^*_1$, $\tW^*_2$ and~$\tW_3$, is achieved in much the same way as for the Bernoulli random graph,
though the details are more complicated.
With $\tW\usp_{\jj,s} := m_0^{-s}(\tZ^+_{\jj,s} + \tZ^+_{3,s})$, $\jj=1,2$,
the exponent
\eqa
  \lefteqn{ m_0^{-2r_n} Z(r,s_1,s_2) }\non\\
  &&\Eq m_0^{-2r_n}\{ Z_{1,r} (\tZ^+_{2,s_1} + \tZ^+_{3,s_1}) + Z_{3,r}(\tZ^+_{2,s_1} + \tZ^+_{3,s_1} + \tZ^+_{1,s_2})
            + Z_{2,r} (\tZ^+_{1,s_2} + \tZ^+_{3,s_2})\}  \non\\
    &&\Eq   m_0^t W^*_{1,r}m_0^{t_1}\tW\usp_{2,s_1} 
        +  m_0^t W^*_{2,r}m_0^{t_2}\tW\usp_{1,s_2}  
          - m_0^{-2r_n}Z_{3,r}\tZ^+_{3,s_2} \phantom{XXXXXXX}  \non \\
    &&\Eq   m_0^{u_1-1} W^*_{1,r}\tW\usp_{2,s_1} 
        +  m_0^{u_2-1} W^*_{2,r}\tW\usp_{1,s_2}  
          - m_0^{-2r_n}Z_{3,r}\tZ^+_{3,s_2}   \ =:\ \baW_n(r,s_1,s_2) \label{DG-baW-def}
\ena
is approximated accurately enough, if $m\th \le \sqrt{m_0}$, by
its limit; \adbr{see Remark~\ref{DG-thm-rk}.  The limit,}
using \Ref{DG-W-lims}, \Ref{DG-Wsp-lim}, \Ref{DG-W3p-lim} and~\Ref{DG-rss-def}, is given by
\[
    (m_0^{u_1-1}W^*_1\tW^*_2 + m_0^{u_2-1}W^*_2\tW^*_1)\Bmmio.
\]
However, if $m\th > \sqrt{m_0}$, in order to obtain the same asymptotic accuracy of approximation as
in~\Ref{DG-main-bnd-1}, it is necessary in addition to approximate
\[
    m_0^{-2r_n}Z_{3,r}\tZ^+_{3,s_2} \Eq (m\th/m_0)^{2r_n} (m\th)^{u_2-1}W_{3,r}\tW^+_{3,s_2}
\]
using $W_3$ and~$m\th\tW_3/(m\th-1)$ in place of $W_{3,r}$ and~$\tW^+_{3,s_2}$, respectively, obtaining
$\hW(u_1,u_2,\e_n)$, \adbr{defined in~\Ref{DG-hW-def},} as an approximation to $\baW_n(r,s_1,s_2)$, or, equivalently,
$(m_0/\chi_n^2) \hW(u_1,u_2,\e_n)$ as an approximation to $n^{-1}m_0 Z(r,s_1,s_2)$ in~\Ref{DG-main-bnd-1}.

To make the comparison, write $T_l(x) := \sum_{i=0}^{l-1}x^{-i}$, and, for $j = 1,2$, define
\eq
    X_l\uj(s) \Def  T_l(m_0)\tW^*_j + \sum_{i=l}^s m_0^{-i}\tW^*_{j,s-i}; \quad 
    V_l\uj(s) \Def  T_l(m_0)\tW^*_{j,s-l} + \sum_{i=l}^s m_0^{-i}\tW^*_{j,s-i}, \label{DG-XV-def}
\en
then setting
\eq\label{DG-XV3-def}
   X_l\uh(s) \Eq  T_l(m\th)\tW_3 + \sum_{i=l}^s m_0^{-i}\tW_{3,s-i}; \quad 
    V_l\uh(s) \Def  T_l(m\th)\tW_{3,s-l} + \sum_{i=l}^s m_0^{-i}\tW_{3,s-i}.
\en
Note that $X_0\uj(s) = \sum_{i=0}^s m_0^{-i}\tW^*_{j,s-i} = \tW\usp_{j,s}$ and that $X_{s+1}\uj(s) = T_{s+1}(m_0)\tW^*_j$,
and similarly that $X_0\uh(s) = \sum_{i=0}^s (m\th)^{-i}\tW_{3,s-i} = \tW^+_{3,s}$ and that $X_{s+1}\uh(s) = T_{s+1}(m\th)\tW_3$.
Note further that, for $j = 1,2$ \adbr{and $0 \le l \le s$},
\eq\label{DG-XV-diffs}
   X_{l+1}\uj(s) - X_l\uj(s) \Eq m_0^{-l}(\tW^*_j - \tW^*_{j,s-l}); \quad 
   X_{l}\uj(s) - V_l\uj(s) \Eq T_l(m_0)(\tW^*_j - \tW^*_{j,s-l}),
\en
and that
\eq\label{DG-XV3-diffs}
   X_{l+1}\uh(s) - X_l\uh(s) \Eq (m\th)^{-l}(\tW_3 - \tW_{3,s-l}); \quad 
   X_{l}\uh(s) - V_l\uh(s) \Eq T_l(m\th)(\tW_3 - \tW_{3,s-l}).
\en
Using these expressions, \adbr{and recalling \Ref{DG-hW-def} and~\Ref{DG-baW-def},}
the difference $\baW_n(r,s_1,s_2) - \hW(u_1,u_2,\e_n)$ can be written as
\eqa
\lefteqn{\baW_n(r,s_1,s_2) - \hW(u_1,u_2,\e_n)} \non\\
 &&\Eq \{m_0^{u_1-1} W^*_{1,r}\tW\usp_{2,s_1}  +  m_0^{u_2-1} W^*_{2,r}\tW\usp_{1,s_2} - m_0^{-2r_n}Z_{3,r}\tZ^+_{3,s_2}\} \non\\
 &&\qquad\qquad\mbox{} - \{m_0^{u_1-1} W^*_{1,r}\tW\usp_{2,s_1}  +  m_0^{u_2-1} W^*_{2,r}\tW\usp_{1,s_2}
                     - \e_n (m\th)^{u_2-1}W_{3,r} \tW^+_{3,s_2} \} \label{DG-limit-1}\\
  &&\quad\mbox{} + \{m_0^{u_1-1} W^*_{1,r}\tW\usp_{2,s_1}  +  m_0^{u_2-1} W^*_{2,r}\tW\usp_{1,s_2}
                     - \e_n (m\th)^{u_2-1}W_{3,r} \tW^+_{3,s_2} \} \non\\
  &&\qquad\qquad\mbox{} - \{m_0^{u_1-1} W^*_1\tW\usp_{2,s_1}  +  m_0^{u_2-1} W^*_2\tW\usp_{1,s_2}
                     - \e_n (m\th)^{u_2-1}W_3 \tW^+_{3,s_2} \} \label{DG-limit-2}\\
  &&\quad\mbox{} + \sum_{l=0}^{s_2} \bigl\{ m_0^{u_1-1} W^*_1 \adbr{(X_{l+s_1-s_2}\ut(s_1) - X_{l+1+s_1-s_2}\ut(s_1))}
                  + m_0^{u_2-1} W^*_2(X_l\ui(s_2) - X_{l+1}\ui(s_2)) \non\\
    &&\qquad\qquad\mbox{}              - \e_n(m\th)^{u_2-1}W_3(X_l\uh(s_2) - X_{l+1}\uh(s_2))\bigr\}
             \label{DG-limit-3}\\
  &&\quad\mbox{}      + \adbb{\sum_{l=0}^{s_1-s_2-1}} m_0^{u_1-1} W^*_1 (X_l\ut(s_1) - X_{l+1}\ut(s_1)) \label{DG-limit-4}\\
  &&\quad\mbox{}      - m_0^{u_1-1}W^*_1\tW^*_2\sum_{i\ge s_1+1}m_0^{-i}  - m_0^{u_2-1}W^*_2\tW^*_1 \sum_{i\ge s_2+1}m_0^{-i}\non\\
   &&\qquad\qquad\mbox{}               + \e_n(m\th)^{u_2-1}W_3\tW_3 \sum_{i \ge s_2+1}(m\th)^{-i}.\label{DG-limit-5}
\ena
This representation is the basis for the estimates to come.
Defining
\eq\label{DG-X*-def}
   X_l^* \Def m_0^{u_1-1} W^*_1 \adbr{X_{l+s_1-s_2}\ut(s_1)}  + m_0^{u_2-1} W^*_2 X_l\ui(s_2) - \e_n (m\th)^{u_2-1} W_3 X_l\uh(s_2),
            \quad 0 \le l \le s_2,
\en
we can write
\eqa
  \D(r,s_1,s_2)   
  &:=& \exp\{-n^{-1}m_0 Z(r,s_1,s_2)\} - \exp\{-(m_0/\chi_n^2)\hW(u_1,u_2,\e_n)\} \non\\
   &=& \sum_{l=1}^5 \D\ul(r,s_1,s_2),\phantom{XXXXXXXXXXXXXXXXXXX} \label{DG-0}
\ena
where
\eqa
  \lefteqn{\D\ui(r,s_1,s_2)} \non\\
      &:=& \exp\{-(m_0/\chi_n^2)(m_0^{u_1-1} W^*_{1,r}\tW\usp_{2,s_1} +  m_0^{u_2-1} W^*_{2,r}\tW\usp_{1,s_2})
           - n^{-1} m_0 Z_{3,r}\tZ^+_{3,s_2}\} \phantom{XXXXXXXX}\non \\
   &&\mbox{} - \exp\{-(m_0/\chi_n^2)(m_0^{u_1-1} W^*_{1,r}\tW\usp_{2,s_1} +  m_0^{u_2-1} W^*_{2,r}\tW\usp_{1,s_2}
                   - \e_n (m\th)^{u_2-1} W_{3,r}\tW^+_{3,s_2})\},\label{DG-1}
\ena
corresponding to the difference~\Ref{DG-limit-1};
\eqa
 \lefteqn{\D\ut(r,s_1,s_2)} \non\\
 &:=& \exp\{-(m_0/\chi_n^2)(m_0^{u_1-1} W^*_{1,r}\tW\usp_{2,s_1} +  m_0^{u_2-1} W^*_{2,r}\tW\usp_{1,s_2}
                      - \e_n (m\th)^{u_2-1} W_{3,r}\tW^+_{3,s_2})\} \phantom{XXXX} \non \\
   &&\mbox{}  -  \exp\{-(m_0/\chi_n^2)(m_0^{u_1-1} W^*_1 \tW\usp_{2,s_1} +  m_0^{u_2-1} W^*_2 \tW\usp_{1,s_2}
                        - \e_n (m\th)^{u_2-1} W_3\tW^+_{3,s_2})\}, \label{DG-2}
\ena
corresponding to the difference~\Ref{DG-limit-2};
\eq\label{DG-3}
  \D\uh(r,s_1,s_2) \Def \sum_{l=0}^{s_2}\bigl\{\exp\{-(m_0/\chi_n^2)X_l^*\} - \exp\{-(m_0/\chi_n^2)X_{l+1}^*\}\bigr\},
\en
corresponding to the difference~\Ref{DG-limit-3};
\eqa
 \D\uf(r,s_1,s_2) &:=& \adbb{\sum_{l=0}^{s_1-s_2-1}}\bigl\{\exp\{-(m_0/\chi_n^2)m_0^{u_1-1} W^*_1 X_l\ut(s_1)\} \non\\
    &&\qquad\qquad\mbox{}             - \exp\{-(m_0/\chi_n^2)m_0^{u_1-1} W^*_1 X_{l+1}\ut(s_1)\}\bigr\}, \phantom{XXX} \label{DG-4}
\ena
corresponding to the difference~\Ref{DG-limit-4}; and
\eqa
 \lefteqn{\D\uv(r,s_1,s_2) }\non\\
 &&\Def   \exp\Bigl\{-(m_0/\chi_n^2)\bigl(m_0^{u_1-1} T_{s_1+1}(m_0) W^*_1 \tW^*_2
                     +  m_0^{u_2-1} T_{s_2+1}(m_0) W^*_2 \tW^*_1 \non\\
 &&\qquad\qquad\qquad\mbox{}                    - \e_n (m\th)^{u_2-1} T_{s_2+1}(m\th) W_3\tW^+_3\bigr)\Bigr\} \non \\
   &&\qquad\mbox{}     - \exp\Bigl\{-(m_0/\chi_n^2)\Bigl( (m_0^{u_1-1} W^*_1 \tW^*_2 + m_0^{u_2-1} W^*_2 \tW^*_1)\Bmmio \non\\
 &&\qquad\qquad\qquad\quad\mbox{}       - \e_n (m\th)^{u_2-1} W_3\tW^+_3\Bmmith \Bigr) \Bigr\},   \label{DG-5}
\ena
corresponding to the difference~\Ref{DG-limit-5}.
Our aim is now to bound $|\ex\{\D(r,s_1,s_2)\}|$, which we do by taking each of the $|\ex\{\D\ul(r,s_1,s_2)\}|$ in turn.

First, it is immediate, using~\Ref{DG-Z3-mean} and the inequality $|e^{-x} - e^{-x'}| \le |x-x'|$ in $x,x' \ge 0$, that
\eqs
    |\ex\{\D\ui(r,s_1,s_2)\}| &\le& n^{-1} m_0\ex|Z_{3,r}\tZ^+_{3,s_2}| \non\\[1ex]
    &\le& (s_2+1) \begin{cases}
             n^{-1}m_0        & \text{ if } m\th \le 1; \\[2ex]
             \frac{m_0}{\chi_n^2} \Bigl(\frac{m\th}{m_0}\Bigr)^{2r_n} (m\th)^{u_2-1}   & \text{ if } 1 < m\th \le \sqrt{m_0}.
          \end{cases}
\ens
Recalling that we have $-2r_n < u_2 \le \half \log n/\log m_0$,
it follows that, if $0 \le m\th \le \sqrt{m_0}$, then
\eq\label{DG-1a}
    |\ex\{\D\ui(r,s_1,s_2)\}| \Le c\ui(m,\th)m_0^{u_2} n^{-1/2}\log n,
\en
where~$c\ui(m,\th)$ is a suitable constant; for $m\th > \sqrt{m_0}$, we have $\D\ui(r,s_1,s_2) = 0$.
\ignore{
and, for $\g' := \log(m\th)/\log m_0$,
\eq\label{DG-alpha-def}
    \a \Def \begin{cases}
             1 - \g'/2  & \text{ if }  m\th \le 1; \\ 
             1 - \g'    & \text{ if } m\th > 1.    
            \end{cases}
\en
}

 The argument for bounding $|\ex\{\D\uv(r,s_1,s_2)\}|$ is similarly direct,
 \adbr{using the explicit formula for the tail of a geometric sum, as well as~\Ref{DG-rss-def},
and yields}
\eqa
   \adbb{|\ex\{\D\uv(r,s_1,s_2)\}|}
   &\le&  \frac{m_0}{\chi_n^2}\,\Bigl\{\Bmmio(m_0^{u_1-s_1-2} + m_0^{u_2-s_2-2}) + \e_n\Bmmith (m\th)^{u_2-s_2-2} \Bigr\}\non\\
                             &\le& c\uv(m,\th)n^{-1/2} m_0^{u_1/2}, \label{DG-5a}
\ena
for a suitable constant~$c\uv(m,\th)$;  here, for the final inequality, \adbr{if $m\th > \sqrt{m_0}$\,,} we use
\eq\label{DG-epsn-power-bnd}
     \e_n(m,\th)(m\th)^{-r_n} \Eq m_0^{-r_n} (m\th/m_0)^{r_n} \Le m_0 n^{-1/2}.
\en

For $|\ex\{\D\ut(r,s_1,s_2)\}|$, we argue along the lines used for~\Ref{B-diff-1}, using the inequality~\Ref{ineq-1} with
$x =y= X(r,s_1,s_2)$ and $x' = X'(r,s_1,s_2)$, where
\[
      X(r,s_1,s_2) \Def (m_0/\chi_n^2)(m_0^{u_1-1} W^*_{1,r}\tW\usp_{2,s_1} +  m_0^{u_2-1} W^*_{2,r}\tW\usp_{1,s_2}
                                       - \e_n (m\th)^{u_2-1} W_{3,r} \tW^+_{3,s_2})
\]
and
\[
    X'(r,s_1,s_2) \Def (m_0/\chi_n^2)(m_0^{u_1-1} W^*_1 \tW\usp_{2,s_1} +  m_0^{u_2-1} W^*_2 \tW\usp_{1,s_2}
                               - \e_n (m\th)^{u_2-1} W_3 \tW^+_{3,s_2}).
\]
Observe that $\ex\bigl\{\exp\{-X(r,s_1,s_2)\}(X(r,s_1,s_2) - X'(r,s_1,s_2))\bigr\} = 0$, because
the processes $W^*_{1,\cdot}$,
$W^*_{2,\cdot}$ and~$W_{3,\cdot}$ are $(\cF^Z_i,\,i\ge0)$ martingales, and independent of~$\tZ$, and
$W^*_{1,r}$, $W^*_{2,r}$ and~$W_{3,r}$
are $\cF^Z_r$-measurable.  This, using \Ref{DG-W-diff-square}  and~\Ref{DG-Wsum-square}, 
yields
\eqa
   \lefteqn{|\ex\{\D\ut(r,s_1,s_2)\}|} \non\\
   &&\Le \frac12 \ex\{(X(r,s_1,s_2) - X'(r,s_1,s_2))^2\} \non\\
   &&\Le \frac32 (m_0/\chi_n^2)^2 \bigl(m_0^{2(u_1-1)} \ex\{(W^*_{1,r} - W^*_1)^2\}\ex\{(\tW\usp_{2,s_1})^2\} \non\\
     &&\qquad\qquad\qquad\qquad\mbox{}          + m_0^{2(u_2-1)} \ex\{(W^*_{2,r} - W^*_2)^2\}\ex\{(\tW\usp_{1,s_2})^2\} \non\\
     &&\qquad\qquad\qquad\qquad\qquad\mbox{}  + \e_n^2 (m\th)^{2(u_2-1)}\ex\{(W_{3,r} - W_3)^2\}\ex\{(\tW^+_{3,s_2})^2\}\bigr) \non\\
   &&\Le c\ut(m,\th) n^{-1/2}m_0^{3u_1/2}, \label{DG-2a}
\ena
for a suitable constant~$c\ut(m,\th)$, again using~\Ref{DG-epsn-power-bnd}.

For $|\ex\{\D\uh(r,s_1,s_2)\}|$, we bound the absolute value of each term
\[
   \ex\{\D^{(3,l)}(r,s_1,s_2)\} \Def \ex\bigl\{\exp\{-(m_0/\chi_n^2)X_l^*\}\bigr\}
             - \ex\bigl\{\exp\{-(m_0/\chi_n^2)X_{l+1}^*\}\bigr\}
\]
in the sum separately, using an argument analogous to that used for~\Ref{B-diff-1}. We use the inequality~\Ref{ineq-1} with
$x = \tX(r,s_1,s_2)$, $x' = \tX'(r,s_1,s_2)$ and $y = V(r,s_1,s_2)$, where
\[
   \tX(r,s_1,s_2) \Def (m_0/\chi_n^2)X_l^*;\qquad \tX'(r,s_1,s_2) \Def (m_0/\chi_n^2)X_{l+1}^*,
\]
and, recalling \Ref{DG-XV-def} and~\Ref{DG-XV3-def},
\[
    V(r,s_1,s_2) \Def (m_0/\chi_n^2)\bigl(m_0^{u_1-1} W^*_1 V_{l+s_1-s_2}\ut(s_1) + m_0^{u_2-1} W^*_2 V_l\ui(s_2)
                                       - \e_n (m\th)^{u_2-1} W_3 V_l\uh(s_2)\bigr),
\]
measurable with respect to $\cF^Z \vee \cF^{\tZ}_{s_2-l}$.  Note also that, from~\Ref{DG-X*-def},
\Ref{DG-XV-diffs} and \Ref{DG-XV3-diffs},
\eqa
     \lefteqn{\tX(r,s_1,s_2) - \tX'(r,s_1,s_2)} \non \\
     &=& (m_0/\chi_n^2)\bigl\{m_0^{u_1-1} W^*_1 \adbr{(X_{l+s_1-s_2}\ut(s_1) - X_{l+1+s_1-s_2}\ut(s_1))} +
                                 m_0^{u_2-1} W^*_2 (X_l\ui(s_2) - X_{l+1}\ui(s_2)) \non \\
      &&\qquad\qquad\qquad\mbox{}                           - \e_n (m\th)^{u_2-1} W_3 (X_l\uh(s_2) - X_{l+1}\uh(s_2))\bigr\} \non \\
          &=&   (m_0/\chi_n^2)\Bigl\{\adbr{m_0^{-l-s_1+s_2}} m_0^{u_1-1} W^*_1 (\tW^*_2 - \adbr{\tW^*_{2,s_2-l}})
                   +  m_0^{-l} m_0^{u_2-1} W^*_2(\tW^*_1 - \tW^*_{1,s_2-l}) \non \\
       &&\qquad\qquad\qquad\mbox{}       - \e_n(m\th)^{-l + u_2-1} W_3(\tW_3 - \tW_{3,s_2-l})\Bigr\},\label{DG-tXd-diff}
\ena
\adbr{having conditional expectation~$0$, given $\cF^Z \vee \cF^{\tZ}_{s_2-l}$,}
and that
\eqa
   \lefteqn{\tX(r,s_1,s_2) - V(r,s_1,s_2)} \non \\
   &=& (m_0/\chi_n^2)\bigl\{m_0^{u_1-1} W^*_1\adbr{(X_{l+s_1-s_2}\ut(s_1) - V_{l+s_1-s_2}\ut(s_1))}
             + m_0^{u_2-1} W^*_2 (X_l\ui(s_2) - V_l\ui(s_2)) \non \\
     &&\qquad\qquad\qquad\mbox{}         - \e_n (m\th)^{u_2-1} W_3 (X_l\uh(s_2) - V_l\uh(s_2))\bigr\}\non \\
   &=& (m_0/\chi_n^2)\Bigl(T_{l+s_1-s_2}(m_0) 
               m_0^{u_1-1} W^*_1\adbr{(\tW^*_2 - \tW^*_{2,s_2-l})}
                          + T_l(m_0) m_0^{u_2-1} W^*_2(\tW^*_1 - \tW^*_{1,s_2-l}) \non \\
     &&\qquad\qquad\qquad\mbox{}          - \e_nT_l(m\th)(m\th)^{u_2-1} W_3 (\tW_3 - \tW_{3,s_2-l})\Bigr),\label{DG-tXV-diff}
\ena
again having conditional expectation~$0$, given $\cF^Z \vee \cF^{\tZ}_{s_2-l}$.
Now, from \Ref{DG-W-diff-square},  \Ref{DG-tXd-diff} and~\Ref{DG-tXV-diff},
we have
  \begin{align}
     &m_0^l\ex\{(\tX(r,s_1,s_2) - \tX'(r,s_1,s_2))^2\} \non\\
     &\Le 3(m_0/\chi_n^2)^2
             \Bigl\{2m_0^{2(u_2-1)-l}\Bmmio^2 m_0^{-s_2+l}
                    + m_0^l\e_n^2(m\th)^{2(u_2-1-l)}\Bmmith^2(m\th)^{-s_2+l}\Bigr\} \non\\
            &\Le 3 m_0^{2u_2-s_2+2} \biggl(\frac2{(m_0-1)^2} + \frac{\bone_{\{m\th > \sqrt{m_0}\}}}{(m\th-1)^2}\biggr) \non\\
            &\ =:\  c\uh_1(m,\th)m_0^{2u_2-s_2} ,\label{ADB-A1}
   \end{align}
since, from \Ref{DG-epsn-def}, for $l \le s_2$ and $m\th > \sqrt{m_0}$,
\[
     \e_n^2 (m\th/m_0)^{2u_2 -l - s_2} \Le (m\th/m_0)^{2(2r_n + u_2 - r_n - t_2)}
          \Eq (m\th/m_0)^{2(r_n + t + 1)},
\]
from \Ref{ADB-A2} and~\Ref{DG-rss-def}, and $r_n + t + 1 \ge 0$. By similar arguments,
and using $u_1 \ge u_2$, it follows that
    \begin{align}
     &m_0^{-l}\ex\{(\tX(r,s_1,s_2) - V(r,s_1,s_2))^2\} \non\\ &\Le 3(m_0/\chi_n^2)^2
             \Bigl\{\Bmmio^3 m_0^{-s_2-2}(m_0^{2u_1} + m_0^{2u_2})
                    + m_0^{-l}\e_n^2\Bmmith^3 (m\th)^{2(u_2-1)-s_2+l} \Bigr\} \non\\
            &\Le 3 m_0^{2u_1-s_2+2} \biggl(\frac{2m_0}{(m_0-1)^3}
                     + \frac{m\th\bone_{\{m\th > \sqrt{m_0}\}}}{(m\th-1)^3}\biggr) \non\\
            &\ =:\  c\uh_2(m,\th)m_0^{2u_1-s_2} .   \label{ADB-A3}
    \end{align}
Thus, from  \Ref{ineq-1}, \Ref{ADB-A1} and~\Ref{ADB-A3},
\eqs
   \lefteqn{|\ex\{\D^{(3,l)}(r,s_1,s_2)\}|} \\
   &\le& \frac12 \ex\{(\tX(r,s_1,s_2) - \tX'(r,s_1,s_2))^2\} \\
   &&\qquad\qquad\qquad\mbox{}    + \ex\{|\tX(r,s_1,s_2) - \tX'(r,s_1,s_2)|\,|\tX(r,s_1,s_2) - V(r,s_1,s_2)|\} \\[1ex]
   &\le& \frac12  c\uh_1(m,\th)m_0^{2u_2-s_2} + m_0^{u_1+u_2-s_2}\sqrt{c\uh_1(m,\th) c\uh_2(m,\th)} \\
       &\le& c\uh(m,\th) n^{-1/2} m_0^{3u_1/2},
\ens
for a suitable constant~$c\uh(m,\th)$, again using~\Ref{DG-epsn-power-bnd}.
Adding over $0 \le l \le s_2$, and using~\Ref{DG-rss-def}, this gives
\eq\label{DG-3a}
   |\ex\{\D\uh(r,s_1,s_2)\}| \Le (r_n + t_2 + 1) n^{-1/2} c\uh(m,\th) m_0^{3u_1/2}.
\en
The same argument can be used for $|\D\uf(r,s_1,s_2)|$ as well, leading to a smaller bound for each~$l$, because the elements
involving $m_0^{u_2-1} W^*_2$ and $(m\th)^{u_2-1}W_3$ are not present, giving
\eq\label{DG-4a}
   |\ex\{\D\uf(r,s_1,s_2)\}| \Le  (t_1 - t_2)c\uh(m,\th) n^{-1/2} m_0^{3u_1/2}.
\en

Combining \Ref{DG-1a}--\Ref{DG-4a}, this results in a bound
\eqa
    |\ex\{\D(r,s_1,s_2)\}| &\le&   c\ui(m,\th) m_0^{u_2/2} n^{-1/2}\log n + c\uv(m,\th)n^{-1/2} m_0^{u_1/2} \label{DG-main-bnd-2}\\
    &&\quad\mbox{} +  c\ut(m,\th)n^{-1/2}m_0^{3u_1/2} + 2c\uh(m,\th)(\log n/\log m_0) n^{-1/2} m_0^{3u_1/2}.
       \non
\ena
Together with~\Ref{DG-main-bnd-1}, this establishes the theorem.
\end{proof}

\section*{Acknowledgement}
ADB thanks the mathematics departments of the University of Melbourne and of Monash University, for
their kind hospitality while part of the work was undertaken.
GR was supported in part by EPSRC grants EP/T018445/1, EP/V056883/1, EP/Y028872/1 and EP/X002195/1.
ADB was supported in part by Australian Research
Council project DP220100973.

 \end{document}